\documentclass[12pt]{amsart}

\usepackage{amssymb,amscd,url}
\usepackage{tikz}

\usepackage{enumerate}

\usepackage{graphicx}

\makeatletter
\@namedef{subjclassname@2010}{%
  \textup{2010} Mathematics Subject Classification}
\makeatother

\newtheorem{theorem}{Theorem}
\newtheorem{lemma}[theorem]{Lemma}
\newtheorem{conjecture}[theorem]{Conjecture}
\newtheorem{proposition}[theorem]{Proposition}
\newtheorem{corollary}[theorem]{Corollary}

\theoremstyle{definition}

\theoremstyle{remark}

\newcommand*{\DashedArrow}[1][]{\mathbin{\tikz [baseline=-0.25ex,-latex, dashed,#1] \draw [#1] (0pt,0.5ex) -- (1.3em,0.5ex);}}
\allowdisplaybreaks

\renewcommand{\a}{\alpha}
\renewcommand{\b}{\beta}
\newcommand{\g}{\gamma}
\renewcommand{\d}{\delta}
\newcommand{\e}{\epsilon}

\renewcommand{\l}{\lambda}

\newcommand{\m}{\mu}

\renewcommand{\o}{\omega}
\renewcommand{\r}{\rho}

\newcommand{\s}{\sigma}

\renewcommand{\t}{\tau}
\newcommand{\z}{\zeta}

\newcommand{\D}{\Delta}
\newcommand{\G}{\Gamma}

\newcommand{\Ocal}{{\mathcal O}}

\newcommand{\CC}{\mathbb{C}}

\newcommand{\NN}{\mathbb{N}}
\newcommand{\PP}{\mathbb{P}}
\newcommand{\QQ}{\mathbb{Q}}
\newcommand{\RR}{\mathbb{R}}
\newcommand{\ZZ}{\mathbb{Z}}

\newcommand{\cufr}{\operatorname{cufr}}

\newcommand{\Div}{\operatorname{Div}}

\newcommand{\lcm}{\operatorname{lcm}}

\newcommand{\notdivide}{\nmid}

\newcommand{\ord}{\operatorname{ord}}

\newcommand{\Pfr}{\operatorname{fr}}

\newcommand{\rank}{\operatorname{rank}}

\newcommand{\sgn}{\operatorname{sign}}

\newcommand{\sq}{\operatorname{sq}}
\newcommand{\sqfr}{\operatorname{sqfr}}

\newcommand{\ve}{\mathbf T}

\begin{document}


\baselineskip=17pt



\title[average value of the canonical height]{On the average value of the canonical height in higher dimensional families of elliptic curves}

\author[W. Wong]{Wei Pin Wong}
\address{Mathematics Department, Box 1917
         Brown University, Providence, RI 02912 USA}
\email{wongpin101@math.brown.edu}

\date{\today}

\begin{abstract}
Given an elliptic curve $E$ over a function field $K=\QQ(T_1, \ldots, T_n)$, we study the behavior of the canonical height $\hat{h}_{E_\o}$ of the specialized elliptic curve $E_\o$ with respect to the height of $\o \in \QQ^n$. In this paper, we prove that there exists a uniform non-zero lower bound for the average of the quotient $\frac{\hat{h}_{E_\o}(P_\o)}{h(\o)}$ for all non-torsion $P \in E(K)$.
\end{abstract}



\maketitle
\section{Introduction}
Let $K$ be the function field $\QQ(T_1, \ldots, T_n)$ and $\ve=(T_1, \ldots, T_n)$. Let $E/K$ be an elliptic curve with Weierstrass equation:
$$ Y^2=X^3+A(\ve)X+B(\ve) $$ where by change of variable, we can assume $A(\ve),B(\ve)\in \ZZ[\ve]$ and there's no nonconstant $g(\ve)\in \QQ[\ve]$ such that $$\frac{ A(\ve)}{g(\ve)^4},\frac{B(\ve)}{g(\ve)^6 } \in \ZZ[\ve].$$ We further assume that $E/K$ is not split over $K$, i.e. $E$ is not $K$-birational isomorphic to $E_0 \times_{\QQ} K$ for any elliptic curve $E_0/\QQ$. This implies $A(\ve)$ and $B(\ve)$ cannot be both constant. The discriminant $$\D_E(\ve)=-16(4A^3(\ve)+27B^2(\ve))$$ is a non-zero element in $\ZZ[\ve]$. Let $\QQ^n(\D_E)$ be the set of all $\o =(\o_1, \dots, \o_n) \in \QQ^n$ such that $\D_E(\o)\neq 0$. Thus for every $P \in E(K)$, for $\o$ such that $P_\o:=P(\o)$ is defined, the point $P_\o$ is a rational point on the elliptic curve $E_\o/ \QQ$ defined by the Weierstrass equation $$ Y^2=X^3+A(\o)X+B(\o).$$ We denote the canonical height on $E_\o$ by $\hat{h}_{E_\o}$ and the logarithmic height on $\PP^n_\QQ$ by $h$, i.e. $$ h(\o)=\log H(\o):=\log H([1,\o_1,\ldots, \o_n]),$$ where $$H([\nu_0,\ldots, \nu_n])=\max_i \{ |\nu_i| \}
,\ \text{ if } \nu_i \in \ZZ \ \text{ and }\gcd (\nu_0, \ldots, \nu_n)=1.$$ To ease the notation, we will denote
$\displaystyle ||\nu ||:=\max_i \{ |\nu_i| \}$ for any $\nu \in \ZZ^n$. We prove the following theorems about the average value of $ \frac{\hat{h}_{E_\o}(P_\o)}{h(\o)}$.

\begin{theorem} \label{T1}  With notation as above, let
$$\QQ^n_{B}(\D_E):=\{ \o \in \QQ^n \ | \ 1 <H(\o) \leq B  \text{ and }\D_E(\o) \neq 0 \},$$and 
$$E(K)_{nt}:=\{ P\in E(K)\ | \ P \text{ non-torsion } \}.$$
Then there exists an $L_1 >0$ depending only on $\D_E$, such that for all $P$ in $E(K)_{nt}$, let 
$$\QQ^n_{B}(\D_E,P):=\{ \o \in \QQ^n_{B}(\D_E) \ | \ P_\o \text{ is defined }\}, $$
we have
$$\underline{Ah}_E^\QQ(P):= \liminf_{B \rightarrow \infty} \frac{1}{\# \QQ^n_{B}(\D_E,P)}\sum_{\o \in \QQ^n_{B}(\D_E,P)} \frac{\hat{h}_{E_\o}(P_\o)}{h(\o)} \geq L_1. $$
\end{theorem}

When $n=1$, Silverman proved in \cite{SHS} that $$ \lim_{\substack{ \o \in \bar{\QQ}^n \\ h(\o) \rightarrow \infty}}\frac{\hat{h}_{E_\o}(P_\o)}{h(\o)}=\hat{h}_E(P),$$ where $\hat{h}_E(P)$ is the canonical height of $P$ in $E/K$. One would like to obtain a similar result for general $n$ but by a simple observation this limit cannot exist for $n \geq 2$. This is because we can restrict the $\o$ to lie on a particular algebraic curve $\g$ for $h(\o)$ tends to infinity, reducing this to the case of $n=1$, but now the limit obtained will depend on $P_\g$ and the elliptic curve $E_\g$ in which it lies. For illustration, consider the elliptic curve $$E/\QQ(S,T) \ : \ Y^2=X^3-S^2X+T^2$$ and  $P=(S,T) \in E(\QQ(S,T))$. If we restrict $\o$ to $\g: \ S=0$, a simple calculation shows that $P_\g=(0,T)$ is a torsion point on $E_\g(\QQ(T)): \ Y^2=X^3+T^2.$ Thus the limit of $ \frac{\hat{h}_{E_\o}(P_\o)}{h(\o)}$ is zero when $h(\o)$ tends to infinity by restriciting $\o \in \g$. On the other hand, if we restrict $\o$ on the curve $\g': \ S=T$, $P_{\g'}=(T,T)$ is in a basis of $E_{\g'}(\QQ(T)) : \ Y^2=X^3-T^2X+T^2$ (this is an example given in \cite{SDS}). Thus Silverman's theorem implies a non-zero limit of the quotient when $h(\o)$ tends to infinity by restriciting $\o \in \g'$. In fact this limit is $\frac{1}{6}$. One can also look at the restriction $T=1$ (resp. $S=1$) and get the limit of the quotient equal to $\frac{1}{2}$ (resp. $\frac{1}{3}$).
 
Since the limit of the quotient $ \frac{\hat{h}_{E_\o}(P_\o)}{h(\o)}$ fails to exist in general for $n \geq 2$, we turn our attention to look at the average of the quotient:
$$  Ah^\QQ_E(P)_B:=\frac{1}{\# \QQ^n_{B}(\D_E,P)}\sum_{\o \in \QQ^n_{B}(\D_E,P)} \frac{\hat{h}_{E_\o}(P_\o)}{h(\o)}.$$
Following the idea of Silverman, we would like make the following conjecture:

\begin{conjecture}\label{Conj} With the same setting as Theorem \ref{T1}, for any $P\in E(K)$,  $$\lim_{B \rightarrow \infty}Ah^\QQ_E(P)_B = \hat{h}_E(P).$$
\end{conjecture}
\noindent The case $n=1$ for this conjecture is true, which follows trivially from Silverman's theorem and Ces\`aro mean theorem. However, proving this conjecture for $n \geq 2$ appears to be difficult, so we first check whether the conjecture even makes sense, i.e., if the limit of the average exists as a function of $P \in E(K)$, does it satisfy the properties of canonical height function ( \cite{SAE} Chapter VIII, Theorem 9.3 or \cite{LFD} Chapter 5)?

One such property is that $\hat{h}_E$ is a quadratic form. By linearity of average, it's straightforward that the limit of $Ah^\QQ_E(-)_B$, if it exists, is a quadratic form too. Another important propety of the canonical height on $E/K$ is that 
$\hat{h}_E(P)=0$ if and only if $P$ is in the subgroup generated by torsion points and the image of $K/\QQ$-trace of $E$ (\cite{LFD} Chapter 6, Theorem 5.4). Since we assume $E$ is not split over $K$, then the $K/\QQ$-trace is of dimension zero, which means it's the trivial group and hence its image in $E$ is the identity (\cite{CCK} Example 2.2). In other words, if $E$ is not split over $K$, then 
\begin{equation}
\hat{h}_E(P)=0 \quad \text{ if and only if } \quad P \text{ is a torsion point.} \label{torsion zero}
\end{equation}

So we investigate property (\ref{torsion zero}) for the limit inferior of $Ah^\QQ_E(P)_B$. We shall prove that the limit inferior of $Ah^\QQ_E(P)_B$ is zero if and only if $P$ is a torsion point of $E(K)$. The if part is trivial as if $P$ is a torsion point of $E(K)$, then $P_\o$ is a torsion point of $E_\o(\QQ)$ and so the average is always zero. It turns out the other direction is also true. We will first prove this by looking at the average over $\ZZ^n$, which is Proposition \ref{p1}. 

\begin{proposition}\label{p1} With notation as above, we further let 
$$\ZZ^n_{B}(\D_E):=\{ \nu \in \ZZ^n \ | \ 1 < ||\nu||\leq B  \text{ and }\D_E(\nu) \neq 0 \}.$$Then there exists an $L_2 >0$ depending only on $\D_E$, such that for all $P$ in $E(K)_{nt}$, let $$\ZZ^n_{B}(\D_E,P):=\{ \nu \in \ZZ^n_{B}(\D_E) \ | \ P_\nu \text{ is defined }\}, $$we have
$$\underline{Ah}_E^\ZZ(P):= \liminf_{B \rightarrow \infty} \frac{1}{(2B)^n}\sum_{\nu \in \ZZ^n_{B}(\D_E,P)} \frac{\hat{h}_{E_\nu}(P_\nu)}{h(\nu)} \geq L_2. $$
\end{proposition}

Proposition \ref{p1} is the key tool used to prove Theorem \ref{T1} via a standard inclusion-exclusion argument. Notice that Proposition \ref{p1} and Theorem \ref{T1} state something stronger: there exists a uniform non-zero lower bound of $\underline{Ah}_E^\ZZ(-)$ and $\underline{Ah}_E^\QQ(-)$ for all non-torsion $P$ in $E(K)$. One might think that the uniform lower bound is expected once we proved that $\underline{Ah}_E^\ZZ(P)>0$ and $\underline{Ah}_E^\QQ(P)>0$ for $P$ in $E(K)_{nt}$, due to the fact that $E(K)_{nt}$ is finitely generated and $\hat{h}_{E_\o}$ can be extended to a positive definite quadratic form on $E_\o(\QQ)\otimes_\ZZ \RR$. At the level of $E_\o(\QQ)$, one can get a uniform lower bound of $\hat{h}_{E_\o}$ on the lattice $E_\o(\QQ)_{nt}\subset E_\o(\QQ)\otimes_\ZZ \RR$ in terms of the canonical height of a nice basis of $E_\o(\QQ)_{nt}$. (\cite{LFD} Chapter 5, Theorem 7.7 and Corollary 7.9). However, at the average level, it is not obvious at all whether one can find a basis $\{P_i\}_{i\in I}$ of $E(K)_{nt}$ such that the specialization $\{P_i(\o)\}_{i\in I}$ is always a nice basis in the image of specialization $(E(K)_{nt})_\o \subseteq E_\o(\QQ)_{nt}$ for all $\o$. Our proofs produce the uniform lower bounds without exploiting these facts. 

We will postpone the proofs of Proposition \ref{p1} and Theorem \ref{T1} to Section \ref{Proof of p1} and \ref{Proof of T1} respectively. On the other side, we also prove that the limit superior of $Ah^\QQ_E(P)_B$ is finite.

\begin{theorem} \label{Up}With the same hypothesis as in Theorem \ref{T1} and for any $P \in E(K)$, there exists a constant $U_{P}$ depending only on $P$, such that
$$ \hat{h}_{E_\o}(P_\o) \leq U_P (1+h(\o))$$
for all $\o \in \QQ^n_{B}(\D_E,P)$. Consequently, we have
$$\overline{Ah}_E^\QQ(P):= \limsup_{B \rightarrow \infty} \frac{1}{\# \QQ^n_{B}(\D_E,P)}\sum_{\o \in \QQ^n_{B}(\D_E,P)} \frac{\hat{h}_{E_\o}(P_\o)}{h(\o)} < \infty. $$
\end{theorem}In fact Theorem \ref{Up} is true in a more general setting as stated in the following theorem:

\begin{theorem}\label{T2} Let $k$ be a number field, let $S$ and $A$ be nonsingular, irreducible, projective varieties defined over $k$, and let $\pi : A \rightarrow S$ be a flat morphism defined over $k$ so that the generic fiber $A_\eta$ of $\pi$ is an abelian variety over $k(S)$. Let
$$S^0:=\{\o \in S(k) \ | \  A_\o \text{ is a non-singular abelian variety defined over }k \}. $$
Fix a divisor $D \in \text{Div}_{\overline{k}}(A)$. For each $\o\in S^0$, let $D_\o \in \text{Div}(A_\o)$ be any divisor in the restriction of the divisor class of $D$ to $A_\o$ and the corresponding canonical height be $\hat{h}_{A_\o,D_\o}$. Fix a projective embedding $i : S \subset \PP^n$,
then for any $P \in A_\eta(k(S))$, there exists a constant $c_0$ depending on $h_{A,D}$, $D$, $i$ and $P$ such that
$$ \hat{h}_{A_\o,D_\o}(P_\o) < c_0(1+h(i(\o)))$$ for all $\o \in S^0$ with $P_\o$ is defined. As a consequence, if we let
$$S^0_B(P):=\{\o \in S^0 \ | \ 1<H(i(\o))<B, \ P \text{ is defined.} \}, $$ then
$$\overline{Ah}^\QQ_{A_\eta}(P):= \limsup_{B \rightarrow \infty} \frac{1}{\# S^0_B(P)}\sum_{\o \in S^0_B(P)} \frac{\hat{h}_{A_\o,D_\o}(P_\o)}{h(i(\o))} < \infty. $$
\end{theorem}
Theorem \ref{T2} is easier to prove than Theorem \ref{T1}, so we will prove this theorem first in Section \ref{proof of T2}. After that we will prove Theorem \ref{Up} in Section \ref{proof of Up} by a similar fashion. 

The behavior of $\hat{h}_{E_\o}(P_\o)$ for $n=1$ is well studied in the literature in a more general setting of an abelian variety $A$ defined over a function field $k(C)$ of a non-singular projective curve $C$ over a number field $k$. In fact this is the original setting in \cite{SHS} where Silverman proved
$$ \lim_{\substack{ t \in C(\bar{k}) \\ h(t) \rightarrow \infty}}\frac{\hat{h}_{A_t}(P_t)}{h(t)}=\hat{h}_A(P).$$
For the special case where $A=E$ is an elliptic surface, Tate \cite{TVC} obtained a stronger result by showing that 
$$\hat{h}_{E_t}(P_t)= \hat{h}_E(P)h(t)+O_P(\sqrt{h(t)}+1)$$ and if $C=\PP^1$, the error is only $O_P(1)$. This stronger result was extended to the general case of abelian varieties by Lang (\cite{LFD} Chapter 12, Section 5) under the assumption that the N\'eron model of the generic fiber has a good completion. In \cite{CVL}, Call reproved Lang's result using a theorem on canonical heights and further discussed cases where the good completion assumption may be weakened or eliminated. Readers can consult Chapter III of \cite{SAT} for a nice introduction and other results on elliptic surfaces. 

Although the behavior of $\hat{h}_{E_\o}(P_\o)$ for $n\geq 2$ is not yet well studied in the literature, we know something about the density of $\o$ such that $\hat{h}_{E_\o}(P_\o)=0$, i.e. $P_\o$ is torsion. Again, this is known in the setting of an abelian variety $A$ defined over a function field $k(V)$ of a variety $V$ over a number field $k$. In \cite{MSF}, Masser proved that for a finitely generated subgroup $\G$ of $A$ the specialization homomorphism 
$$\s_\o : \G \rightarrow A_\o(k(\o)) $$ is injective ``almost always'' for $\o \in V(\bar{k})$. 

\section{Proof of Theorem \ref{T2}} \label{proof of T2}

Notice that it suffices to prove that the quotient $ \frac{\hat{h}_{A_\o,D_\o}(P_\o)}{h(i(\o))}$ is bounded above uniformly for all $\o \in S^0_B(P)$. This is an immediate consequence of Theorem A of \cite{SHS}, due to Silverman and Tate. In effect, with the given hypothesis in Theorem \ref{T2} and further let $h_{A,D}$ be the Weil height (defined up to equivalence) corresponding to $D$, Theorem A says that there exists a constant $c$ depending on $D$ and $A$, so that for all $P \in A_\eta(K)$
$$\left|\hat{h}_{A_\o,D_\o}(P_\o)-h_{A,D}(P_\o)\right|< c h(i(\o))+O(1), $$ where $O(1)$ depends on the choice of particular Weil heights $h_{A,D}$ and the embedding $i$. So we turn the problem into estimating $h_{A,D}(P_\o)$.

We remind the reader about the definition of $h_{A,D}$. If $D \in \text{Div}_{\overline{k}}(A)$ is very ample, then choose an embedding 
$$ \phi_D : A \rightarrow \PP^m_{\overline{k}} $$ correspnding to the linear system $|D|$ and $h_{A,D}$ is defined by 
\begin{align*}
h_{A,D} : A(\overline{k}) &\longrightarrow \RR \\
p &\longmapsto h(\phi_D(p)).
\end{align*}
For a general divisor $D \in \text{Div}_{\overline{k}}(A)$, write $D=X-Y$, where $X,Y\in \text{Div}_{\overline{k}}(A)$ are very ample divisors, and define
$$h_{A,D}(p):=h_{A,X}(p)-h_{A,Y}(p). $$
For any $P \in A_\eta(K)$, it defines a rational map
\begin{align*}
\psi_P: S &\DashedArrow[->,]  A \\
\o &\longmapsto P_\o.
\end{align*}
So we have
\begin{align*}h_{A,D}(P_\o)&= h_{A,X}(P_\o)-h_{A,Y}(P_\o)\\
&=h(\phi_X(\psi_P(\o)))-h(\phi_Y(\psi_P(\o)))\\
&\leq h(\phi_X(\psi_P(\o))),
\end{align*} where $f_X:=\phi_X\circ\psi_P$ is a rational map from $S$ to $\PP^m$. By using the triangle inequality of absolute values of $k$, one can show the following standard property of height on projective space (\cite{LFD} Chapter 4, Lemma 1.6):
$$ h(f_X(\o))\leq d h (i(\o))+c_1$$ for some constant $c_1$ and $d$ that depend on $f_X$ only. 
Finally, by applying Theorem A, we get
\begin{align*}
\hat{h}_{A_\o,D_\o}(P_\o)&\leq h_{A,D}(P_\o) + c h(i(\o))+O(1)\\
&\leq d h (i(\o))+c_1 + c h(i(\o))+O(1),
\end{align*}
which is the first part of the theorem. Since the set of points of bounded height in $\PP^n(k)$ is finite, there's a non-zero lower bound (which depends on $k$) for $h(i(\o))>0$. We obtain our desired uniform upper bound for $\frac{\hat{h}_{A_\o,D_\o}(P_\o)}{h(i(\o))}$ by dividing the inequality above by $h(i(\o))>0$ and hence proved the second part of the theorem.

\section{Proof of Theorem \ref{Up}} \label{proof of Up}
We remark that Theorem \ref{Up} doesn't follow trivially from Theorem \ref{T2} even if we can find a nonsingular irreducible projective variety $\mathcal{E}/\QQ$ and a flat morphsim $\pi : \mathcal{E} \longrightarrow \PP^n$ with generic fiber $\mathcal{E}_\eta$ isomorphic to $E/K$. This is because it is not true in general that we can find a divisor $D \in \Div_{\overline{\QQ}}(\mathcal{E})$ such that 
$$h_{\mathcal{E},D}(p)=h_{\PP^1}([x(p),1]) +O(1),$$ due to the fact that the $X$-coordinate map 
\begin{align*}
\phi: \mathcal{E} &\DashedArrow[->,]  \PP^1 \\
p &\longmapsto [x(p),1]
\end{align*}is just a rational map in general. By mimicking the idea of the proof of Theorem A in \cite{SHS}, one can overcome this by blowing up $\mathcal{E}$ and extending $\phi$ to a morphism. However, we found a more direct and elementary proof for Theorem \ref{Up}, which is the one that we are going to present. 

Using just the definition of height on elliptic curves and triangle inequality of absolute values of $\QQ$, we first prove that there exist positive constants $c_1, c_2$ such that for all $\o \in \QQ^n(\D_E)$ and all $p \in E_\o(\QQ)$, we have
\begin{equation} h_{E_\o}([2]p)-4h_{E_\o}(p) \leq c_1h(\o)+c_2. \label{ATA} \end{equation}
Recall that $E_\o$ is defined by the Weierstrass equation:
$$ Y^2=X^3+A(\o)X+B(\o).$$ For any $p=(x,y) \in E_\o(\QQ)$, we may assume $[2]p \neq \textsl{O}_{E_\o}$ or otherwise inequality (\ref{ATA}) is trivially true for any positive $c_1,c_2$. The duplication formula gives
$$x([2]p)=\frac{x^4-2A(\o)x^2-8B(\o)+A(\o)^2}{4x^3+4A(\o)x+4B(\o)}. $$Thus, we have
\begin{align}
&H_{E_\o}([2]p) \notag \\
&:=H\left([x([2]p),1]\right) \notag \\
&=H\left([x^4-2A(\o)x^2-8B(\o)x+A(\o)^2,4x^3+4A(\o)x+4B(\o)]\right) \notag \\
&\leq 4H\left([1,-2A(\o),-8B(\o),A(\o)^2,4,4A(\o),4B(\o)]\right)H\left([x,1]\right)^4\notag \\
&\leq 4 N_{A,B} H(\o)^{d_{A,B}} H_{E_\o}(p)^4 \label{ATA1}
\end{align}where the inequalities are obtained by triangle inequality of absolute values of $\QQ$. The constant $N_{A,B}$ depends on the coefficients and the number of monomials of $A$ and $B$, whereas $d_{A,B}$ is the maximum of $\deg A^2$ and $\deg B$. Inequality (\ref{ATA}) is obtained by taking natural logarithm of (\ref{ATA1}).

Now, we use Tate's telescoping sum trick to prove an analogy of Theorem A in \cite{SHS}:
\begin{align}
\hat{h}_{E_\o}(p)-h_{E_\o}(p) &= \sum_{n=0}^\infty \frac{1}{4^{n+1}}\left(h_{E_\o}([2^{n+1}]p)-4h_{E_\o}([2^n]p)\right) \notag \\
&=\sum_{n=0}^\infty \frac{1}{4^{n+1}}\left(h_{E_\o}([2]\circ[2^{n}]p)-4h_{E_\o}([2^n]p)\right)\notag \\
&\leq \sum_{n=0}^\infty \frac{1}{4^{n+1}} \left( c_1h(\o)+c_2\right) \qquad (\text{Using (\ref{ATA})}) \notag \\
&= \frac{c_1}{3}h(\o)+\frac{c_2}{3}.\label{Up1}
\end{align}
Finally, given any $P=(x(\ve), y(\ve)) \in E(K)$, the $X$-coordinate of $P$ defines a rational map 
\begin{align*}
\psi_P: \PP^n &\DashedArrow[->,]  \PP^1 \\
\o &\longmapsto [x(\o),1].
\end{align*}Just like in the proof of Theorem \ref{T2}, the standard property of height on projective space gives
\begin{align}h_{E_\o}(P_\o)&:=h([x(\o),1]) \notag \\
&= h(\psi_P(\o)) \notag \\
&\leq dh(\o)+c_3 \label{Up2}
\end{align}for some constants $d, c_3$ that depend on $\psi_P$ only. We get our conclusion of Theorem \ref{Up} by combining (\ref{Up1}) and (\ref{Up2}).

\section{Lemmas} \label{Lemmas}
Besides some results on elliptic curves over $\QQ$, the proof of Proposition \ref{p1} requires several non-trivial facts about polynomials with integer coefficients. In this section, we will state these results and give complete proofs with appropriate references. We remind the reader that we continue to use all the notations that we have defined previously. In addition, for the specialized elliptic curve $E_\o$, let $\D_{E_\o}=\D_E(\o)$ and $\D_{E_\o}^{\min}$ be the discriminant and minimum discriminant of $E_\o/\QQ$ respectively. Also, for any UFD $R$, whenever we say $P_1, \ldots, P_n$ are relatively prime in $R[\ve]$, we always mean that $P_1, \ldots, P_n$ don't have a common irreducible factor in $R[\ve]$.

\begin{lemma}\label{lowerB KN} There exists an absolute constant $C_1>0$ such that the following holds. Let $k \geq 4$ be an integer, $N_k:=\text{lcm}(1,2,3,\ldots,k)$ and suppose $\nu \in \ZZ^n$ so that $\D_E(\nu)$ is non-zero and $k^{\text{th}}$-power-free (abbreviated as $k$-free for the rest of the paper). Then for any non-torsion point $q \in E_\nu(\QQ)$, we have
$$\hat{h}_{E_\nu}(q) > \frac{C_1}{N_k^2}\log |\D_{E_\nu}^{\min}|.  $$  
\end{lemma}
\begin{proof}
We make use of a weakened form of a conjecture of Serge Lang proved by Silverman in section 4 of \cite{SLB}. We apply it to a non-torsion point $q\in E_\nu(\QQ)$ such that $q$ is in $$(E_\nu)_0(\QQ_p):=\{ q\in E_\nu(\QQ_p) \ | \ q \pmod{p} \text{ is non-singular} \}$$ for every prime $p$ in $\QQ$. This is possible by Kodaira-N\'eron Theorem (\cite{SAT} Chapter VII, Theorem 6.1) which implies that the order of $E_\nu(\QQ_p)/(E_\nu)_0(\QQ_p)$ is either $\ord_p(\D_{E_\nu}^{\min})$ or at most $4$. So if $\D_E(\nu)$ is $k$-free, we have $\ord_p(\D_{E_\nu}^{\min}) \leq k$ and thus $[N_k]q$ is in $(E_\nu)_0(\QQ_p)$ for all $p$ with the choice of $N_k:=\text{lcm}(1,2,3,\ldots,k)$. Then the special case of the conjecture gives 
$$ \hat{h}_{E_\nu}([N_k]q) > C_1\log |\D_{E_\nu}^{\min}|,$$ for an absolute constant $C_1>0$. Using the fact $\hat{h}_{E_\nu}$ is a quadratic form will complete the proof.
\end{proof}

\begin{lemma}\label{small o}$$\sum_{\substack{\nu \in \ZZ^n \\ 1< ||\nu|| \leq B}} \frac{1}{\log ||\nu||}=o(B^n),$$where the implicit constant in the small $o$ depends only on $n$.

\end{lemma}
\begin{proof} In this proof, the implicit constants of all the big $O$'s depend only on $n$. By symmetry of each quadrant in $\ZZ^n$, we have 
\begin{align*}
\sum_{\substack{\nu \in \ZZ^n \\ 1< ||\nu|| \leq B}} \frac{1}{\log ||\nu||}&=O\left(\sum_{1 < x_1\leq x_2\leq \ldots \leq x_n\leq B} \frac{1}{\log x_n}\right)\\
&=O\left( \int_2^B \int_0^{x_{n}} \cdots \int_0^{x_2} \frac{1}{\log x_n} dx_1 \cdots dx_n\right)\\
&=O\left( \int_2^B \frac{1}{(n-1)!}\frac{x_n^{n-1}}{\log x_n} dx_n\right)\\
&=O\left( \int_2^{\sqrt{B}} \frac{t^{n-1}}{\log t} dt+\int_{\sqrt{B}}^B \frac{t^{n-1}}{\log t} dt\right)\\
&=O\left(\frac{B^n}{\log B}\right).
\end{align*}
\end{proof}

\begin{lemma} \label{Mason}  Let $k,m,r \in \NN$ satisfy
\[
  \frac{1}{k} + \frac{1}{m} + \frac{1}{r} \le 1.
  \] If $P,Q,R \in \CC[\ve]$ satisfy $ P^k+Q^m=R^r,$ then either $P,Q,R$ are all constant or else they are not relatively prime.

\end{lemma}
\begin{proof}
We first prove the case $n=1$, which is an immediate consequence of Mason--Stothers theorem ({\cite{LA} Chapter IV, Theorem 7.1} or \cite{SPI} Theorem 1.1). 
Suppose to the contrary that $P,Q,R$ are not all constant and relatively prime, then by Mason--Stothers theorem, we have
\begin{align*} \max \{ k\deg P, \  m\deg Q, \ r\deg R \} +1 &\leq \# \text{distinct roots of }P^kQ^mR^r\\
&\leq \deg P+\deg Q+\deg R.
\end{align*} Without lose of generality, suppose $k\deg P \geq m\deg Q$, which implies $k\deg P \geq r\deg R$, so the inequality above becomes
\begin{align*} k\deg P +1 &\leq \deg P+\frac{k}{m}\deg P+\frac{k}{r}\deg P
\end{align*}
$$\Rightarrow 1 \leq \left(\frac{1}{k}+\frac{1}{m}+\frac{1}{r}-1 \right)k\deg P \leq 0, $$ which is absurd. 

Now, let $P,Q,R \in \CC[\ve]$ satisfy the hypothesis of the lemma. Suppose $P,Q,R$ not all constant and relatively prime. Without lose of generality, we can assume the degrees of $T_n$ in $P,Q$ are at least $1$. We will make use of some standard results about the resultant of two polynomials in $R[x]$, where $R$ is a UFD. These results eventually boil down to linear algebra (\cite{KAA} Chapter VIII, Theorem 8.1). Consider $P,Q$ as element in $\CC[T_1, \ldots, T_{n-1}][T_n]$ and let $f \in \CC[T_1, \ldots, T_{n-1}]$ be the resultant of $P,Q$ with respect to the variable $T_n$. Then there exist non-zero $u,v \in \CC[\ve]$ with $\deg_{T_n}u < \deg_{T_n}Q $ and $\deg_{T_n}v < \deg_{T_n}P $ such that
$$uP+vQ=f.$$ 
Since $P,Q$ have no common factor in $\CC[\ve]$, $f$ cannot be identically zero. We can choose $y:=(y_1,\ldots,y_{n-1}) \in \CC^{n-1}$ such that $f(y) \neq 0$ and $P(y,T_n)$ is nonconstant. Then $P_y(T_n):=P(y,T_n)$ and $Q_y(T_n):=Q(y,T_n)$ are relatively prime in $\CC[T_n]$ and $P_y(T_n)$ is nonconstant. So we get relatively prime $P_y,Q_y,R_y \in \CC[T_n]$ such that not all are constant and satisfies the hypothesis of the lemma for $n=1$, which is impossible as we have shown previously. 
\end{proof}

\begin{lemma} \label{Mason Co}  Let $k,m,r \in \NN$ satisfy
\[
  \frac{1}{k} + \frac{1}{m} + \frac{1}{r} \le 1.
\]
Let $\ell = \lcm(k,m)$ and $g = \gcd(k,m)$, and assume that $\ell | r$. Let $P,Q,R \in \CC[\ve]$ be polynomials with $R \ne 0$ that satisfy
\[
  P^k + Q^m = R^r.
\]
Then there exists $\alpha_1,\alpha_2 \in \CC$ such that
$$                   P=\a_1 R^{\frac{m}{g}\frac{r}{\ell}} \quad \text{ and } \quad Q=\a_2 R^{\frac{k}{g}\frac{r}{\ell}}. $$

\end{lemma}
\begin{proof} The case where $P,Q,R$ are all constant is trivial. So suppose $P,Q,R$ are not all constant. We let $S:=R^{\frac{r}{\ell}}$ and we have
\begin{equation}P^k+Q^m=S^\ell. \label{Eq PQS} \end{equation} 
Let $G_1, \ldots, G_s$ be the distinct irreducible factors of $PQS$ and write 
$$P=\a\prod_i G_i^{a_i}, \quad Q=\b\prod_i G_i^{b_i}, \quad S=\g\prod_i G_i^{c_i}$$ with $\a,\b,\g \in \CC$. Then we can rewrite the equality (\ref{Eq PQS}) as
\begin{equation}\a^k\prod_i G_i^{a_ik}+\b^m\prod_i G_i^{b_im}=\g^\ell\prod_i G_i^{c_i\ell}. \label{eq factor G}\end{equation}
We claim that $a_ik=b_im=c_i\ell$ for all $i$.
Notice that we cannot have one exponent of $G_i$ in equation (\ref{eq factor G}) that is strictly less than the other two,
otherwise by dividing by the least power $G_i$ factor, we get a contradiction. So two of the exponents of $G_i$ in equation (\ref{eq factor G}) are equal and at most equal to the third one. We divide equation (\ref{eq factor G}) by $G_i$ with the common lower exponent and we do this for all $i$.  Using the fact that $\ell=\lcm(k,m)$, the resulting equation can be written in the form
$$ P_1^k+Q_1^m=S_1^\ell, $$ where $P_1,Q_1,S_1$ are either all constant or relatively prime. Notice that the former case corresponds to our claim $a_ik=b_im=c_i\ell$ for all $i$ and we are going to prove that this must be the case. Since $\frac{1}{k}+\frac{1}{m}+\frac{1}{r}\leq 1$, without lose of generality, $k \geq 2$ and $m \geq 3$ and one easily verifies that $\ell:=\text{lcm}(k,m) \geq 6$ except for the cases $(k,m,\ell)=(3,3,3), (2,4,4), (4,4,4), (5,5,5)$. So we always have $\frac{1}{k}+\frac{1}{m}+\frac{1}{\ell}\leq 1$ and hence we can apply lemma \ref{Mason} on $P_1,Q_1,S_1$ to conclude that they are all constant. So we have $$P=\a_1 S^\frac{\ell}{k}=\a_1 S^\frac{m}{g} \quad \text{ and } \quad Q=\a_2 S^\frac{\ell}{m}=\a_2 S^\frac{k}{g} .$$ Substituting back $S=R^{\frac{r}{\ell}}$ completes the proof.     
\end{proof}

To avoid heavy notation in the proofs below, we denote $$\ZZ_B:=\ZZ \cap [-B,B]$$ and for any $F\in \ZZ[\ve]$,
$$\r_F(m):=\{\nu\in(\ZZ/m\ZZ)^n \ | \ F(\nu)\equiv 0 \pmod{m}\}, $$
$$||F||:=\max\{|c| \ | \ c \text{ is a coefficient of }F \}. $$
Note: 
\begin{enumerate}
\item By abuse of notation, the symbol $\equiv$ used in the proofs of lemmas \ref{Np} and \ref{Np2 nonsq} has three different meanings depending on the context. When $f$ is an element of $\ZZ[\textbf{x}]$, $f \equiv 0$ means $f$ is the zero polynomial. The notation $f \equiv 0$ in $\ZZ/p\ZZ[\textbf{x}]$ means the reduction mod $p$ of $f$ is the zero polynomial in $\ZZ/p\ZZ[\textbf{x}]$. If we evalaute $f$ at $x$ and $f(x)$ is an integer, the notation $f(x) \equiv 0 \pmod{p}$ means $p$ divides $f(x)$.
\item By definition, a polynomial $F\in \ZZ[\ve]$ consists the information of its domain. Thus, if an implicit constant in the big $O$ or small $o$ notation is said to be dependent on $F$, that means that it depends on $\deg F$ and $n$ as well.
\end{enumerate}

\begin{lemma} \label{Np} Let $F \in \ZZ[\ve]$ with total degree $d \geq 1$. Then for all prime $p$ bigger than $||F||$, we have
$$N_p(F,B):=\#\{ \nu \in \ZZ_B^n \ | \ F(\nu) \equiv 0 \pmod{p} \}=O\left(\frac{B^n}{p}+B^{n-1}\right), $$where the implicit constant in the big $O$ depends only on $n$ and $d$.

\end{lemma}
\begin{proof}In this  proof, the implicit constants of all the big $O$'s depend only on $n$ and $d$. We prove by induction on $n$. For $n=1$, with the condition on $p$, $F \not \equiv 0 \in \ZZ/p\ZZ[T_1]$. So $$N_p(F,B)\leq \r_F(p)\left(\frac{2B}{p}+1\right) \leq d \left(\frac{2B}{p}+1\right).$$ 
Now let $F \in \ZZ[\ve]$ and for $y \in \ZZ^{n-1}$, $F_y(T_n):=F(y,T_n) \in \ZZ[T_n].$ The condition $F_y \equiv 0 \text{ in } \ZZ/p\ZZ[T_n]$ becomes a bunch (at most $d$) of polynomials of degree at most $d$ in $\ZZ[T_1,\ldots,T_{n-1}]$ equal zero mod $p$. Thus, by induction hypothesis, $$\#\{ y \in \ZZ_B^{n-1} \ | \ F_y \equiv 0 \text{ in } \ZZ/p\ZZ[T_n]  \}=O\left(\frac{B^{n-1}}{p}+B^{n-2}\right).$$ So we get
\begin{align*}
N_p(F,B)&=\sum_{\substack{ y \in \ZZ_B^{n-1} \\ \hidewidth F_y \equiv 0 \text{ in } \ZZ/p\ZZ[T_n]  \hidewidth}} N_p(F_y,B)+\sum_{\substack{ y \in \ZZ_B^{n-1} \\  \hidewidth F_y \not \equiv 0 \text{ in } \ZZ/p\ZZ[T_n] \hidewidth }} N_p(F_y,B)\\
&\leq O\left(\frac{(2B)^{n-1}}{p}+(2B)^{n-2}\right)(2B+1)+(2B+1)^{n-1} d \left(\frac{2B}{p}+1\right)\\
&=O\left(\frac{B^n}{p}+B^{n-1}\right).
\end{align*}
\end{proof}

\begin{lemma} \label{Np2 nonsq} \label{k-free} Suppose $F(\ve) \in \ZZ[\ve]$ has total degree $d\geq 1$ and has no repeating irreducible factor in $\ZZ[\ve]$. Then except for finitely many prime $p$, we have
\begin{align*}
N_{p^2}(F,B)&:=\#\{ \nu \in \ZZ_B^n \ | \ F(\nu) \equiv 0 \pmod{p^2} \}\\&=O\left(\frac{B^n}{p^2}+B^{n-1}\right), 
\end{align*}whenever $p \leq B$. The implicit constant in the big $O$ depends only on $F$.
\end{lemma}
\begin{proof}In this  proof, the implicit constants of all the big $O$'s depend only on $F$, which includes $d$ and $n$. Since we allow finitely many exception on $p$, we can assume $F$ is primitive, i.e. the content of $F$ is $1$. For $n=1$, if $p \notdivide $Disc$(F)\neq 0$, then $N_{p^2}(F,B)\leq d\left(\frac{2B}{p^2}+1\right).$ 
Now let $F \in \ZZ[\ve]$ and for $y \in \ZZ^{n-1}$, $F_y(T_n):=F(y,T_n) \in \ZZ[T_n].$ Let $\textbf{Y}:=(T_1, \ldots, T_{n-1})$. By Gauss' lemma for UFDs, the fact that $F$ has no repeating irreducible factor in $\ZZ[\ve]=\ZZ[\textbf{Y}][T_n]$ implies the same holds in $\QQ(\textbf{Y})[T_n]$. So $D(\textbf{Y}):=\text{Disc}(F_\textbf{Y})\not \equiv 0 \in \ZZ[\textbf{Y}]$. Also, for $p$ bigger than $||D||$, $D(\textbf{Y})\not \equiv 0 \in \ZZ/p\ZZ[\textbf{Y}]$.
We write $F_\textbf{Y}(T_n)=a_d(\textbf{Y})T_n^d+\ldots+a_0(\textbf{Y})$ and we divide into two cases:

\noindent Case 1: gcd$(a_d(\textbf{Y}), \ldots, a_0(\textbf{Y}))=1$. We decompose $N_{p^2}(F,B)$ into the following three sums:
\begin{align*}
\sum_{\substack{  y \in \ZZ_B^{n-1}  \\ \hidewidth D(y) \not \equiv 0 \pmod{p} \hidewidth} } N_{p^2}(F_y,B)+\sum_{\substack{ y \in \ZZ_B^{n-1} \\ \hidewidth F_y \not \equiv 0   \text{ in } \ZZ/p\ZZ[T_n] \hidewidth \\ \hidewidth D(y) \equiv 0 \pmod{p} \hidewidth }} N_{p^2}(F_y,B)+\sum_{\substack{ y \in \ZZ_B^{n-1} \\ \hidewidth F_y \equiv 0   \text{ in } \ZZ/p\ZZ[T_n] \hidewidth }} N_{p^2}(F_y,B)
\end{align*}
The first sum is trivially bounded by $(2B+1)^{n-1}d\left(\frac{2B}{p^2}+1\right)$. Whereas for the second sum, we apply lemma \ref{Np} on $D(\textbf{Y})$ to get an upper bound $$O\left(\frac{B^{n-1}}{p}+B^{n-2}\right)dp\left(\frac{2B}{p^2}+1\right)=O\left(\frac{B^n}{p^2}+B^{n-1}\right).$$ Lastly, since gcd$(a_d(\textbf{Y}), \ldots, a_0(\textbf{Y}))=1$, if we look at the $y\in(\ZZ/p\ZZ)^{n-1}$ such that $F_y \equiv 0   \text{ in } \ZZ/p\ZZ[T_n]$, either $y$ is a common root in $(\ZZ/p\ZZ)^{n-1}$ of at least two polynomials that are relatively prime in $\ZZ[T_1, \ldots,T_{n-1}]$ or there is no such $y$ because there is only one non-zero $a_i(\textbf{Y})$ and it must be $1$ by assumption of case 1. From the proof of Theorem 3.1 of Poonen in \cite{PSV}, the set of such $y\in(\ZZ/p\ZZ)^{n-1}$ has order $O(p^{(n-1)-2})$ for large $p$. Hence we have for $p\leq B$,
\begin{align*}
\sum_{\substack{ y \in \ZZ_B^{n-1} \\ \hidewidth F_y \equiv 0   \text{ in } \ZZ/p\ZZ[T_n]\hidewidth }} N_{p^2}(F_y,B)&\leq O(p^{n-3})\left(\frac{2B}{p}+1\right)^{n-1}(2B+1)\\
&=O(p^{n-3})O\left(\frac{B^{n-1}}{p^{n-1}}+\frac{B^{n-2}}{p^{n-2}}\right)(2B+1)\\
&=O\left(\frac{B^n}{p^2}+B^{n-1}\right)
\end{align*}and we are done.

\noindent Case 2: gcd$(a_d(\textbf{Y}), \ldots, a_0(\textbf{Y}))=g(\textbf{Y}) \neq 1$. Then $F(\textbf{T})=A(\textbf{T})g(\textbf{Y})$ for some nonconstant $A(\textbf{T}) \in \ZZ[\ve].$ Then
$$
N_{p^2}(F,B)
=\sum_{\substack{ y \in \ZZ_B^{n-1} \\ \hidewidth g(y) \equiv 0 \pmod {p^2}\hidewidth}} O(2B)+\sum_{\substack{ y \in \ZZ_B^{n-1} \\ p||g(y)}} N_{p}(A_y,B)+\sum_{\substack{ y \in \ZZ_B^{n-1} \\ p \notdivide g(y)}} N_{p^2}(A_y,B),
$$
where $p||g(y)$ means $p|g(y)$ but $p^2 \notdivide g(y)$. Since $g(\textbf{Y})$ does not have repeating irreducible factor in $\ZZ[T_1,\ldots, T_{n-1}]$ too, we use the induction on $n$ to bound the first sum by $O_{d,n}\left(\frac{B^{n-1}}{p^2}+B^{n-2}\right)O(2B)$. As for the third sum, it is trivially bounded by $N_{p^2}(A,B)$ and this reduces to case 1. Finally, we split the middle sum as follows:
$$\sum_{\substack{ y \in \ZZ_B^{n-1} \\ p||g(y)}} N_{p}(A_y,B)=\sum_{\substack{ y \in \ZZ_B^{n-1} \\ p||g(y) \\ A_y \not \equiv 0 \text{ in } \ZZ/p\ZZ[T_n]}} N_{p}(A_y,B)+\sum_{\substack{ y \in \ZZ_B^{n-1} \\ p||g(y) \\ A_y \equiv 0 \text{ in } \ZZ/p\ZZ[T_n]}} O(2B). $$ Using lemma \ref{Np}, we can bound the first sum by 
$$O\left(\frac{B^{n-1}}{p}+B^{n-2}\right)d \left(\frac{2B}{p}+1\right)=O\left(\frac{B^n}{p^2}+B^{n-1}\right).$$ As for the second sum, since $A$ is of case 1, the order of $y \in (\ZZ/p\ZZ)^{n-1}$ such that $A_y \equiv 0 \text{ in } \ZZ/p\ZZ[T_n]$ is $O(p^{(n-1)-2})$ and so the sum is bounded by 
$$O(p^{n-3}) \left(\frac{2B}{p}+1\right)^{n-1}O(2B)=O\left(p^{n-3}\left(\frac{3B}{p}\right)^{n-1} B\right)=O\left(\frac{B^n}{p^2}\right),$$ for large $p\leq B$.
\end{proof}

\begin{lemma} \label{k-free} Suppose $F(\ve) \in \ZZ[\ve]$ has total degree $d$ and has no repeating irreducible factor in $\ZZ[\ve]$. 
Then for all integers $k\gg d$  
we have
$$\# \{ \nu \in \ZZ_B^n \ | \ F(\nu) \text{ is } k\text{-free} \}\sim \g_{k,F}(2B)^n, $$ where $\displaystyle \g_{k,F}:=\prod_{\text{prime }p\in \ZZ}\left(1-\frac{\r_F(p^k)}{p^{nk}}\right)$ is a non-zero convergent Euler product. Here, we adopt the convention that $0$ is $k$-free for all integer $k \geq 2$.
\end{lemma}
\begin{proof}In this  proof, we will introduce some arbitrary constants $\xi, \e>0$, and the implicit constants of all the big $O$'s depend only on $F$, $k$, $\xi$ and $\e$. We adapt the idea of Browning in section 4 of \cite{BPF}. Let $\xi>0$ be a constant and define the following sets:
\begin{align*}
N_\text{fr}^{(k)}:&=\{ \nu \in \ZZ_B^n \ | \ F(\nu) \text{ is } k\text{-free} \}, \\
N_\text{nfr, 1}^{(k)}:&=\left\{ \nu \in \ZZ_B^n \ \left| \parbox{0.6\hsize}{$F(\nu)$ is not  $k$-free, and for all prime $p$ such that $p^k|F(\nu)$, we have $\xi<p\leq B$} \right.\right\},\\
N_\text{nfr, 2}^{(k)}:&=\left\{  \nu \in \ZZ_B^n  \ \left| \ \parbox{0.6\hsize}{$F(\nu)$ is not $k$-free, and for all prime $p$ such that $p^k|F(\nu)$, we have $p>\xi$, and  \\$p^k|F(\nu)$ for some prime $p > B$ }\right.\right\},\\
N^{(k)}_\xi:&=\{ \nu \in \ZZ_B^n \ | \text{ if } p^k|F(\nu) \text{ then } p>\xi \}=N_\text{fr}^{(k)}\sqcup N_\text{nfr, 1}^{(k)} \sqcup N_\text{nfr, 2}^{(k)},\\
M_\text{fr}^{(2)}:&=\left\{ \nu \in \ZZ_B^n \ \left| \begin{array}{ll} F(\nu) \text{ is } k\text{-free, and} \\ p^2|F(\nu) \text{ for some prime } \xi<p \leq B \end{array}\right.\right\},\\
M_\text{nfr}^{(2)}:&=\left\{ \nu \in \ZZ_B^n \ \left| \begin{array}{ll} F(\nu) \text{ is not } k\text{-free, and} \\ p^2|F(\nu) \text{ for some prime }  \xi<p \leq B\end{array}\right.\right\}, \\
M^{(2)}:&=\{ \nu \in \ZZ_B^n \ | \ p^2|F(\nu) \text{ for some prime } \xi<p \leq B\}=M_\text{fr}^{(2)}\sqcup M_\text{nfr}^{(2)}.
\end{align*}
Then obviously we have $M_\text{fr}^{(2)} \subset N_\text{fr}^{(k)}$ and $N_\text{nfr, 1}^{(k)} \subset M_\text{nfr}^{(2)}$ and thus 
$$ \#N^{(k)}_\xi \geq \#N_\text{fr}^{(k)} \geq \#N^{(k)}_\xi - \#M^{(2)} - \#N_\text{nfr, 2}^{(k)}. $$

We first estimate $\#N^{(k)}_\xi$ with the help of M\"{o}bius function $\mu$. Let
$$\r_F(m):=\#\{ \nu \in  (\ZZ/m\ZZ)^n\ | \ F(\nu) \equiv 0 \pmod{m} \}.$$ Then we can write
\begin{align*}
\#N^{(k)}_\xi&=\sum_{ \substack{ h \in \NN \\ p|h\Rightarrow p \leq \xi} }\mu(h)\#\{ \nu \in \ZZ_B^n \ | \ h^k|F(\nu)\}\\
&=\sum_{ \substack{ h \in \NN \\ p|h\Rightarrow p \leq \xi} }\mu(h)\r_F(h^k)\left( \frac{2B}{h^k}+O(1)\right)^n\\
&=\sum_{ \substack{ h\in \NN \\ p|h\Rightarrow p \leq \xi} } \mu(h)\r_F(h^k)\left( \frac{(2B)^n}{h^{kn}}+O\left(\left(\frac{2B}{h^k}\right)^{n-1}+1\right)\right).
\end{align*}Since the summation sums only square-free $h$, the condition $p|h\Rightarrow p \leq \xi$ implies 
$$ h \leq \prod_{p \leq \xi} p = \exp \left( \sum_{p \leq \xi} \log p \right) \leq e^{2\xi},$$where the last inequality is gotten by the prime number theorem. Moreover, it follows from the proof of Theorem 3.2 of Poonen in \cite{PSV} that $\r_F(p^2)=O(p^{2n-2})$ (or we can also deduce this from lemma \ref{Np2 nonsq} with $B=p^2$). Subsequent lifting will lead to $\r_F(p^j)=O(p^{jn-2})$ for $j\geq 2$. Together with the fact that $\r_F$ is multiplicative, we have for square-free $h$ that $\r_F(h^k)=O(h^{nk-2+\e})$ for any $\e>0$. The $\e$ is needed here in order to bound the product of $r_h$ copies of the implicit constant of $O(p^{2n-2})$, where $r_h$ is the number of distinct prime factors of $h$. Using the fact that $h$ is square-free, we have $h\geq r_h !$  and the Stirling's formula will give us the desired bound. With this, we obtain

$$\#N^{(k)}_\xi=(2B)^n\prod_{p \leq \xi}\left(1-\frac{\r_F(p^k)}{p^{nk}} \right)+O\left((2B)^{n-1}e^{2\xi(k-1+\e)}+e^{2\xi(nk-1+\e)}\right). $$ Again because $\r_F(p^k)=O(p^{kn-2})$, the infinite product $$ \g_{k,F}:=\prod_{\text{prime }p\in \ZZ}\left(1-\frac{\r_F(p^k)}{p^{nk}}\right)$$ converges and we have
$$\#N^{(k)}_\xi \geq (2B)^n\g_{k,F}+O(B^{n-1}). $$

Next, by lemma \ref{Np2 nonsq}, we have
\begin{align*}
\#M^{(2)}&\leq \sum_{ \xi< p \leq B}\#\{ \nu \in \ZZ_B^n \ | \ p^2|F(\nu)\}\\
&= \sum_{ \xi< p \leq B }O\left( \frac{B^n}{p^2}+B^{n-1}\right)\\
&\leq c \left( \frac{B^n}{\xi}+\frac{B^n}{\log B}\right)
\end{align*} for some constant $c>0$ depending only on $n$ and $d$. The first term of the upper bound is obtained by integral estimate and the second is by the Prime Number Theorem.

Lastly, $\#N_\text{nfr, 2}^{(k)}=0$ for $B$ big enough. In effect, there exists a constant $C_F>0$ such that $ |F(\nu)| \leq C_F ||\nu||^d$. Thus, for all $B>C_F$, prime $p>B$, $\nu\in \ZZ_B^n$ and $k>d$, we have
$$p^k>B^k\geq B^{d+1}>C_FB^d\geq|F(\nu)|, $$ so no such $p^k$ divides $F(\nu)$.

Combining together all the estimates, we get
$$\#N_\text{fr}^{(k)} \geq \g_{k,F}(2B)^n+ O(B^{n-1})-\frac{c}{\xi}(2B)^n+o_{d,n}(B^n)$$ for $B>\max\{\xi,C_F\}$. The fact that $\r_F(p^k)=O(p^{kn-2})$ implies $\g_{k,F}$ converges and $\g_{k,F}$ is zero if and only if one of its factors is zero. So in order to make $\g_{k,F}>0$, it's sufficient to choose $k$ big enough such that $\r_F(p^{k})<p^{nk}$ for all prime $p$. More explicitly, choose a $\nu_0$ such that $F(\nu_0)\neq 0$ and look at its prime factorization $\displaystyle \prod_i p_i^{\b_i}$ in $\ZZ$. Then any $\displaystyle k> \max_i\{\b_i,d\}$ will do. We fix this $k$ and for any $\l\in(0,1)$, by choosing $\xi$ big enough so that $\frac{c}{\xi} \ll \g_{k,F}$, we get 
$$ \#N_\text{fr}^{(k)} \geq \l \g_{k,F}(2B)^n $$  for $B$ big enough. Since for all $\xi\gg0$, $\#N^{(k)}_\xi \geq \#N_\text{fr}^{(k)}$ and $$\#N^{(k)}_\xi \sim (2B)^n\prod_{p \leq \xi}\left(1-\frac{\r_F(p^k)}{p^{nk}} \right),$$ we get 
$$ \prod_{p \leq \xi}\left(1-\frac{\r_F(p^k)}{p^{nk}} \right) \geq \limsup_{B \rightarrow \infty}\frac{\#N_\text{fr}^{(k)}}{(2B)^n}\geq \liminf_{B \rightarrow \infty}\frac{\#N_\text{fr}^{(k)}}{(2B)^n} \geq  \l \g_{k,F} $$ for all $\xi \gg 0$ and $\l \in (0,1)$. Taking $\l \rightarrow 1$ (which forces $\xi \rightarrow \infty$ ) will complete the proof.

\end{proof}

\begin{corollary} \label{k-free Co} For any $F(\ve) \in \ZZ[\ve]$, then there exists an integer $k_0$, such that for all integer $k \geq k_0$ and for all $\l \in (0,1)$, there exist $B_0>0$ depending on $F,k,\l$ and $c_{k,F}>0$ depending on $F,k$ such that
$$\# \{ \nu \in \ZZ_B^n \ | \ F(\nu) \text{ is } k\text{-free} \}\geq \l c_{k,F}(2B)^n $$ whenever $B \geq B_0$.
\end{corollary}
\begin{proof}Write $\displaystyle  F=\prod_{i=1}^r f_i^{\a_i}$ where $f_i$ are distinct irreducible factors of $F$ in $\ZZ[\ve]$. Let $\displaystyle  f:=\prod_{i=1}^r f_i$ with total degree $d$ and $\displaystyle \a:=\max_i\{\a_i \}$. Now it is immediate by the previous lemma that for all $k\geq k'\a$, where $k'> d$ big enough as in the previous lemma, we obtain our corollary with $c_{k,F}=\g_{k',f}>0$. 
\end{proof}

\begin{lemma}\label{Good density} For any integer $N \geq 2$, we denote $\Pfr_N(m)$ to be the $N^\text{th}$- power-free part of the integer $m$, i.e. the smallest positive integer $\ell$ such that $ \frac{|m|}{\ell}$ is a $N^\text{th}$-power of an integer. Suppose a primitive $F(\ve)\in \ZZ[\ve]$ is not a $p^\text{th}$-power in $\CC[\ve]$ for all prime $p|N$. Then for all $M>2$, we have
$$\#\{ \nu \in \ZZ^n_B\ | \ \left(\Pfr_N(F(\nu))\right)^M> ||\nu||\} \sim (2B)^n.$$
\end{lemma}
\begin{proof} Let the total degree of $F$ be $d$ and hence there exists a constant $C_F \geq 1$ such that $ |F(\nu)|\leq C_F ||\nu||^d.$ Define 
$$S_M(F,B):=\left\{ (\nu, y, z)\in \ZZ^{n+2}\ \left| \begin{array}{lll} ||\nu||\leq B, \\ |y|\leq (C_FB^d)^{\frac{1}{N}}, \quad F(\nu)=y^N z \\ 0<|z|\leq B^{\frac{1}{M}},    \end{array}\right.\right\}. $$
We will prove by induction that
$$\# S_M(F,B) \ll_{F,\e,M,N}C^\e_FB^{(n-1)+\frac{1}{N}+\frac{1}{M}+2d\e}\log B \qquad \forall \e>0.$$
The implicit constants of the big $O$'s and small $o$'s that appear in this proof will depend only on $F$, $\e$, $M$ and $N$. We are going to apply Theorem 15 of Heath-Brown in \cite{HCR}, so we try to use notations that are coherent with it.
For $n=1$, for all $z_0 \in \ZZ$, let
$$N(f_{z_0},B,(C_FB^d)^{\frac{1}{N}}):=\left\{(\nu,y)\in \ZZ^{2} \ \left| \begin{array}{ll} ||\nu||\leq B, \ |y|\leq (C_FB^d)^{\frac{1}{N}},\\ f_{z_0}(\nu,y):=F(\nu)-z_0y^N=0 \end{array}\right.\right\}. $$Then 
\begin{align*}
\# S_M(F,B)=\sum_{0<|z_0|\leq B^{\frac{1}{M}}} \#N(f_{z_0},B,(C_FB^d)^{\frac{1}{N}}). 
\end{align*}Since for all prime $p|N$, $F$ is not a $p^\text{th}$-power in $\CC[T_1]$, the same holds in $\CC(T_1)$. By Capelli's lemma (\cite{LA} Chapter VI, Theorem 9.1), $f_{z_0}(T_1,Y)=F(T_1)-z_0Y^N$ is absolutely irreducible in $\CC(T_1)[Y]$ for all $z_0 \neq 0$. We need this fact for the next step.

Now we apply Theorem 15 of Heath-Brown in \cite{HCR} on $N(f_{z_0},B,(C_FB^d)^{\frac{1}{N}})$. Let $T:=\max\{B^d,C_FB^d \}=C_FB^d $. Then for all $\e>0$, there exists a constant $D=D_{d,\e}$ and $k \in \NN$ with
\begin{align*}
k &\ll_{d,\e}T^\e \exp \left\{ \frac{\log B\log (C_FB^d)^\frac{1}{N}}{\log (C_FB^d)}\right\} \log || f_{z_0}||\\
&\ll_{d,\e} (C_F B^d)^\e B^\frac{1}{N} \log || f_{z_0}||,
\end{align*}such that there exists $\tilde{f_1}, \ldots, \tilde{f_k} \in \ZZ[T_1,Y]$, coprime to $f_{z_0}$ and with degrees at most $D$, such that every $(\nu,y)$ counted by $N(f_{z_0},B,(C_FB^d)^{\frac{1}{N}})$ is a zero of some polynomial $\tilde{f_i}$. By B\'ezout's theorem, the number of points of intersection of curves $\tilde{f_i}=0$ and $f_{z_0}=0$ is bounded by $\deg \tilde{f_i} \cdot \deg f_{z_0}\leq D(d+N).$ This gives immediately 
$$\#N(f_{z_0},B,(C_FB^d)^{\frac{1}{N}}) \ll_{d,\e}D(d+N) (C_F B^d)^\e B^\frac{1}{N} \log || f_{z_0}||.$$ So
\begin{align}
\#S_M(F,B) &\ll_{d,\e,N}\sum_{0<|z_0|\leq B^{\frac{1}{M}}}(C_F B^d)^\e B^\frac{1}{N} \log || f_{z_0}||  \label{n1}\\
&\ll_{d,\e,N}\sum_{0<|z_0|\leq B^{\frac{1}{M}}}(C_F B^d)^\e B^\frac{1}{N} \log B \notag\\
&\ll_{d,\e,N}C_F^\e B^{\frac{1}{N}+\frac{1}{M}+d\e} \log B,\notag
\end{align}where we may choose $B\geq || F||$ and hence $|| f_{z_0}||\leq B$ for $|z_0|\leq B^{\frac{1}{M}}$.
Now we proceed to prove for a general $n \geq 2$. For all $x \in \ZZ^{n-1}$, let $F_x(T_n):=F(x,T_n)\in \ZZ[T_n]$. For all $p|N$, since $F(\ve)$ is not a $p^\text{th}$-power, we look at the $p^\text{th}$-power-free part of $F(\ve)$ in $\ZZ[\ve]$, call it $G_p(\ve)$. In other words, $G_p(\ve)$ is the smallest degree polynomial such that  $\frac{F(\ve)}{G_p(\ve)}$ is a $p^\text{th}$-power in $\ZZ[\ve]$. So $G_p(\ve)=\prod_j G_{p,j}(\ve)^{\b_j}$ where $G_{p,j}$ are distinct irreducible factors and $0<\b_j <p$. Let $g_p(\ve):=\prod_j G_{p,j}(\ve)$, which has no repeated irreducible factor in $\ZZ[\ve]$. Using Gauss' lemma on UFDs and by reindexing if necessary, the discriminant of $g_p(\textbf{X},T_n) \in (\ZZ[\textbf{X}])[T_n]$ is not a zero polynomial in $\ZZ[\textbf{X}]$. So there are at most $O(B^{n-2})$ of $x \in \ZZ^{n-1}_B$ such that $g_{p,x}(T_n):=g_p(x,T_n)$ has repeated irreducible factor in $\ZZ[T_n]$. This will imply that there are at most $O(B^{n-2})$ of $x \in \ZZ^{n-1}_B$ such that $F_x(T_n)$ is a $p^\text{th}$-power in $\ZZ[T_n]$. So we have
\begin{equation}\#S_M(F,B)= \sum_{\substack{x \in \ZZ_B^{n-1} \\ \hidewidth F_x \text{ non-$p^\text{th}$-power} \hidewidth \\ \text{for all }p|N \hidewidth}}\#S_M(F_x,B)+\sum_{\substack{x \in \ZZ_B^{n-1} \\ \hidewidth F_x \text{ $p^\text{th}$-power} \hidewidth \\ \text{for some }p|N }}\#S_M(F_x,B).\label{n>2}\end{equation}
Notice that $|F_x(\nu_n)| \leq (C_FB^d)|\nu_n|^d$ for $x \in \ZZ_B^{n-1}$ and $\deg F_x \leq d$. So using the result from the case $n=1$, we get 
\begin{align*}
&\#S_M(F,B)\\ &\ll_{d,\e,N}(2B)^{n-1}(C_FB^d)^\e B^{\frac{1}{N}+\frac{1}{M}+d\e} \log B + O(B^{n-2} \cdot B \cdot B^{\frac{1}{M}})\\
&\ll_{F,\e,M,N}C_F^\e B^{(n-1)+\frac{1}{N}+\frac{1}{M}+2d\e}\log B,
\end{align*}where for the estimation of the second sum, we use the fact that $y$ is determined (up to sign for the case $N$ is even) once $(\nu_n,z_0)$ is fixed in $F_x.$ When $\e$ is sufficiently small relative to $d$ and $M>2$, we get $\#S_M(F,B)=o(B^n).$ Lastly, define 
$$ \text{Bad}_M(F,B):=\{ \nu \in \ZZ^{n}_B\ | \ (\Pfr_N(F(\nu)))^M \leq ||\nu||\}, $$ which is the complement of $\{ \nu \in \ZZ^n_B\ | \ (\Pfr_N(F(\nu)))^M>||\nu||\}$ in $\ZZ_B^n$. It is a simple exercise to show that $\text{Bad}_M(F,B)$ injects into $S_M(F,B)$ via the map $\nu \longmapsto (\nu,\sqrt[N]{|F(\nu)|/\Pfr_N(F(\nu))} , \sgn(F(\nu))\Pfr_N(F(\nu)))$, hence giving us the lemma.
\end{proof}

\begin{corollary}\label{CorPfree}With the same hypothesis as in the previous lemma and further let $g \in \NN$ such that $0<g\leq B^{\frac{1}{N+2}}$, then for $M$ big enough, we have
$$\#\left\{ \left.\nu \in \ZZ^n_\frac{B}{g}\ \right| \ \left(\Pfr_N(F(\nu))\right)^M> g^{M(N-1)+1}||\nu||\right\} \sim \left(2\frac{B}{g}\right)^n.$$In particular, when $N=2 \text{ or }3$, then any $M> 8$ is admissible.
\end{corollary}
\begin{proof} The proof is just a slight modification of the previous proof, so we will continue using all the notations from the previous proof. Again, the implicit constants of the big $O$'s and small $o$'s in this proof depend only on $F,\e,M$ and $N$. We are going to show that the complement of $$\left\{ \left.\nu \in \ZZ^n_\frac{B}{g}\ \right| \ \left(\Pfr_N(F(\nu))\right)^M> g^{M(N-1)+1}||\nu||\right\},$$ which is 
$$\text{Bad}_M(F,B,g):=\left\{ \left.\nu \in \ZZ^n_\frac{B}{g}\ \right| \ \left(\Pfr_N(F(\nu))\right)^M \leq  g^{M(N-1)+1}||\nu||\right\},$$ has order $o\left(2\frac{B}{g}\right)^n$. Just like the previous proof, $\text{Bad}_M(F,B,g)$ injects into
$$S_M(F,B,g):=\left\{ (\nu, y, z)\in \ZZ^{n+2}\ \left| \begin{array}{lll} ||\nu||\leq \frac{B}{g}, \\ |y|\leq \left(C_F\left(\frac{B}{g}\right)^d\right)^{\frac{1}{N}}, \quad F(\nu)=y^N z \\ 0<|z|\leq g^{N-1+\frac{1}{M}} \left(\frac{B}{g}\right)^{\frac{1}{M}},    \end{array}\right.\right\}. $$ Thus, it suffices to show that $\#S_M(F,B,g)=o\left(2\frac{B}{g}\right)^n$. Comparing $S_M(F,B,g)$ to $S_M(F,B)$ from the previous proof, this boils down to just changing $B$ to $\frac{B}{g}$ and slightly increasing the upper bound for $z$ with a factor of $g^{N-1+\frac{1}{M}}$. So for $n=1$, from $(\ref{n1})$, we have
\begin{align*}
\#S_M(F,B,g) &\ll_{d,\e,N}\sum_{0<|z_0|\leq g^{N-1+\frac{1}{M}} \left(\frac{B}{g}\right)^{\frac{1}{M}}}\left(C_F\left(\frac{B}{g}\right)^d\right)^\e \left( \frac{B}{g}\right)^\frac{1}{N} \log || f_{z_0}||\\
&\ll_{d,\e,N}C_F^\e \left(\frac{B}{g}\right)^{\frac{1}{N}+\frac{1}{M}+d\e}  g^{N-1+\frac{1}{M}}\log B,
\end{align*}where we choose $B \geq|| F||$ and hence $||f_{z_0}||\leq \max \{B, g^{N-1}B^\frac{1}{M} \}\leq B^2$. Since $g \leq B^\frac{1}{N+2}$, we have $\frac{B}{g}\geq g^{N+1}$ and $ B\leq \left( \frac{B}{g}\right)^\frac{N+2}{N+1}< \left( \frac{B}{g}\right)^2$, so 
\begin{align*}\#S_M(F,B,g)&\ll_{d,\e,N}C_F^\e \left(\frac{B}{g}\right)^{\frac{1}{N}+\frac{1}{M}+d\e}  \left(\frac{B}{g}\right)^{\frac{N-1}{N+1}+\frac{1}{(N+1)M}}\log \left(\frac{B}{g}\right),
\end{align*}which is $o\left(2\frac{B}{g}\right)$ for $M$ sufficiently big and $\e$ sufficiently small. Using the same induction argument as in $(\ref{n>2})$, we have for $n \geq 2$, 
\begin{align}
&\#S_M(F,B,g) \notag \\ &\ll_{d,\e,M,N}\left(2\frac{B}{g}\right)^{n-1}\left(C_F\left(\frac{B}{g}\right)^d\right)^\e \left(\frac{B}{g}\right)^{\frac{1}{N}+\frac{1}{M}+d\e+\frac{N-1}{N+1}+\frac{1}{(N+1)M}} \log \left(\frac{B}{g}\right) \notag\\
 &\quad + O\left(\left(\frac{B}{g}\right)^{n-2} \cdot \frac{B}{g} \cdot g^{N-1+\frac{1}{M}} \left(\frac{B}{g}\right)^{\frac{1}{M}}\right) \notag \\
&\ll_{F,\e,M,N}C_F^\e \left(\frac{B}{g}\right)^{n-1+\frac{1}{N}+\frac{1}{M}+2d\e+\frac{N-1}{N+1}+\frac{1}{(N+1)M}} \log \left(\frac{B}{g}\right), \label{last ineq}
\end{align}which is also $o\left(2\frac{B}{g}\right)^n$ for $M$ sufficiently big and $\e$ sufficiently small. We remark that for the case $N=2,3$, the exponent of $\frac{B}{g}$ in (\ref{last ineq}) are $n-1+2d\e+\frac{5}{6}+\frac{4}{3M}$ and $n-1+2d\e+\frac{5}{6}+\frac{5}{4M}$ respectively. For any $M >8$, there exists $\e>0$ such that this exponent is strictly less than $n$, hence giving us corollary \ref{CorPfree} for the case $N=2,3$ and these are the instances where we will apply this corollary. 
\end{proof}
\section{Proof of Proposition \ref{p1}}
\label{Proof of p1}
We keep all the notations as previously defined in this paper. The implicit constants of the big $O$'s and small $o$'s that appear in this proof will depend only on $\D_E, n$ and $P$. The main idea in this proof is to first apply lemma \ref{lowerB KN}. This allows us to get a lower bound of $\hat{h}_{E_\nu}(P_\nu)$ in term of $\D_{E_\nu}^{\min}$, for all ``nice'' $\nu \in \ZZ^n$. Then we try to bound $\D_{E_\nu}^{\min}$ below in term of $\D_{E_\nu}$ and then in term of $h(\nu)$, again for all ``nice'' $\nu$. The nontrivial part of the proof is to show that after we impose again and again certain niceness conditions on $\nu$, this set of of ``nice'' $\nu$ has a positive density in $\ZZ^n_B(\D_E,P)$.   

Fix a big integer $k\geq 4$, which we will specify how big it should be at the end of the proof and let $N_k:=\text{lcm}(1,2,3,\ldots,k)$. Then by lemma \ref{lowerB KN}, there is an absolute constant $C_1>0$ such that for any $P \in E(K)_{nt}$ and for any $\nu \in \ZZ^n_B(\D_E,k,P_{\nu}^{nt})$, where 
$$\ZZ^n_{B}(\D_E,k,P_{\nu}^{nt}):= \{\nu \in \ZZ^n_{B}(\D_E,P) \ | \ \D_E(\nu) \text{ is $k$-free, } P_\nu \in E_\nu(\QQ)_{nt} \},$$ we have 
\begin{align} \label{Eq1}
\sum_{\nu \in \ZZ^n_{B}(\D_E,P)} \frac{\hat{h}_{E_\nu}(P_\nu)}{h(\nu)} &\geq \frac{C_1}{N_k^2}\sum_{\nu \in \ZZ^n_{B}(\D_E,k,P_{\nu}^{nt})} \frac{\log |\D^{\min}_{E_\nu}|}{h(\nu)} \notag  \\
&=\frac{C_1}{N_k^2}\sum_{\nu \in \ZZ^n_{B}(\D_E,k,P_{\nu}^{nt})} \frac{\log |\D^{\min}_{E_\nu}|}{\log ||\nu||}. 
\end{align} We obtain the second line because of the convention that we made earlier : $h(\nu)=\log H([1,\nu_1,\ldots,\nu_n]).$

Next, we claim that $\D_E(\ve)=-16(4A^3(\ve)+27B^2(\ve))$ is never a constant times a twelfth power in $\ZZ[\ve]$, otherwise lemma \ref{Mason Co} says that there exists $g(\ve)\in \QQ[\ve]$ such that $\frac{ A(\ve)}{g(\ve)^4},\frac{B(\ve)}{g(\ve)^6 } \in \ZZ$. So using Gauss' lemma, we can write $\D_E(\ve)=\a (F(\ve))^{a+12b}$, where $\a \in \ZZ$, $F(\ve)$ is primitive in $\ZZ[\ve]$, non-power in $\CC[\ve]$, $a \in \{1,2,3,\ldots, 11 \}$ and $b\in \NN$. Recall that for all primes $p$ in $\QQ$,
$$0\leq \ord_p(\D^{\min}_{E_\nu}) \equiv \ord_p(\D_{E_\nu}) \pmod{12}, $$ so $\ord_p(\D^{\min}_{E_\nu}) \text{ is at least the unique integer in }\{0,1,2,\ldots, 11 \}$ congruent to $\ord_p(\D_{E_\nu}) \pmod{12}$. We split into two cases in order to get a lower bound of $\D^{\min}_{E_\nu}$.\\

\noindent Case 1: $a \neq 4,8$. If we let $\sqfr(m):=\Pfr_2(m)$ and $\sq(m):=\frac{|m|}{\sqfr(m)}$ be the square-free part and square part of an integer $m$, then we have
\begin{align*}
|\D_E(\nu)|=|\a|\left|\sq(F(\nu))\right|^{a+12b}\left|\sqfr(F(\nu))\right|^{a+12b}.
\end{align*}
Notice that for every prime factor $p$ of $\sqfr(F(\nu))$ that is relatively prime to $\a$, its power $\b_p$ in $F(\nu)$ is odd and thus $a \b_p \not \equiv 0 \pmod{12}$ for $a \neq 4,8$. So $p$ is a factor of $\D^{\min}_{E_\nu}$ and we have for $\nu \in \ZZ^n_{B}(\D_E,k,P_{\nu}^{nt})$, 
$$ |\D^{\min}_{E_\nu}|\geq \frac{|\sqfr(F(\nu))|}{| \a|}. $$

\noindent Case 2: $a = 4 \text{ or } 8$. The argument is similar to case 1 except that we look at $\cufr(F(\nu)):=\Pfr_3(F(\nu))$, the cube-free part of $F(\nu)$. Then for every prime factor $p$ of $\cufr(F(\nu))$ that is relatively prime to $\a$, its power $\b_p$ in $F(\nu)$ is not a multiple of $3$ and thus $a \b_p \not \equiv 0 \pmod{12}$ for $a =4 \text{ or } 8$. In fact, $a \b_p \equiv 4 \text{ or } 8 \pmod{12}.$ So again, for $\nu \in \ZZ^n_{B}(\D_E,k,P_{\nu}^{nt})$, we have  
$$ |\D^{\min}_{E_\nu}|\geq \frac{|\cufr(F(\nu))|}{|\a|^2}. $$

Fix $M>2$ and let $$\text{Good}_M(F,B):= \left\{\begin{array}{lll}
                   \left\{ \nu \in \ZZ^n_{B}(\D_E)\ | \ |\sqfr(F(\nu))|^M>||\nu||\right\} & \text{if } a \neq 4,8 \\ \\ \left\{ \nu \in \ZZ^n_{B}(\D_E)\ | \ |\cufr(F(\nu))|^M>||\nu||\right\} & \text{if } a = 4 , 8.
 \end{array}\right.$$ Then from inequality (\ref{Eq1}), we have

\begin{align*}
\sum_{\nu \in \ZZ^n_{B}(\D_E,P)} \frac{\hat{h}_{E_\nu}(P_\nu)}{h(\nu)} 
&\geq \frac{C_1}{N_k^2}\sum_{\nu \in \ZZ^n_{B}(\D_E,k,P_{\nu}^{nt})\cap \text{Good}_M(F,B)} \frac{\log ||\nu||^{\frac{1}{M}}-\log|\a|^2}{\log ||\nu||}\\
&=\frac{C_1}{N_k^2M}\qquad \quad \ \sum_{\hidewidth \qquad \nu \in \ZZ^n_{B}(\D_E,k,P_{\nu}^{nt})\cap \text{Good}_M(F,B)\hidewidth }\ 1 \qquad \qquad+\qquad o(B^n),
\end{align*}
where we use lemma \ref{small o} to bound the sum of the second term. Now we are at the final step of analyzing the asymptotic cardinal of the set $\ZZ^n_{B}(\D_E,k,P_{\nu}^{nt})\cap \text{Good}_M(F,B)$. It is straightforward that $$\# \ZZ^n_{B}(\D_E,P) \sim (2B)^n$$ as the set of points for which $\D_E(\ve)$ vanishes or $P_\nu$ is not defined is of order at most $O(B^{n-1})$. Next, by Mazur's theorem (\cite{SAE} Chapter VIII, Theorem 7.5), the order of $E_\nu(\QQ)_{tor}$ is at most $12$. Hence if $P_\nu$ is torsion, $\nu$ must satisfy one of the twelve algebraic equations of torsion points that depends on $P$. Since $P$ is non-torsion in $E(K)$, none of the twelve equations is identically zero and so
$$\# \{\nu \in \ZZ^n \ | \ H(\nu) \leq B \text{ and } P_\nu \text{ is torsion}  \}=O( B^{n-1}).$$ This gives 
$$\#\ZZ^n_{B}(\D_E,P_{\nu}^{nt}):= \#\{\nu \in \ZZ^n_{B}(\D_E,P) \ | \  P_\nu \in E_\nu(\QQ)_{nt} \} \sim (2B)^n.$$ We now apply corollary \ref{k-free Co} to $\D_E(\ve)$, and we specify that $k$ is big enough such that $c_{k,\D_E}>0$ as in the corollary. Then for any $\l \in (0,1)$ and for $B$ big enough, we get
$$\#\ZZ^n_{B}(\D_E,k,P_{\nu}^{nt}) \geq \l c_{k,\D_E}(2B)^n.$$ Lastly, since $F$ is primitive and is neither a square nor cube in $\ZZ[\ve]$, we use lemma \ref{Good density} to conclude 
$$\#\left(\ZZ^n_{B}(\D_E,k,P_{\nu}^{nt})\cap \text{Good}_M(F,B) \right)\geq \l c_{k,\D_E}(2B)^n$$ and this give us
\begin{align*}
\sum_{\nu \in \ZZ^n_{B}(\D_E,P)} \frac{\hat{h}_{E_\nu}(P_\nu)}{h(\nu)} \geq \frac{C_1}{N_k^2M} \l c_{k,\D_E}(2B)^n+o(B^n)
\end{align*}
This proves Proposition \ref{p1} with a lower bound $\frac{C_1}{N_k^2M} \l c_{k,\D_E}$ for any $\l \in (0,1)$ and $M>2$, hence we can take $L_2=\frac{C_1}{2N_k^2} c_{k,\D_E}$.

\section{Proof of Theorem \ref{T1}}
\label{Proof of T1}
The idea of this proof is to reduce to the case of Proposition \ref{p1}, since a point in $\PP^n_\QQ$ can be represented with integers coordinates. Again, the implicit constants of the big $O$'s and small $o$'s that appear in this proof will depend only on $\D_E, n$ and $P$, unless stated otherwise. Let $\o=\left(\frac{u_1}{v_1}, \ldots, \frac{u_n}{v_n} \right)\in \QQ^n$ in the lowest form and let $\ell_\o:=\lcm (v_1, \ldots, v_n)$. Recall that the Weierstrass equation of $E_\o$ is 
$$Y^2=X^3+A(\o)X+B(\o) ,$$ which might not have integer coefficients. In order to estimate $\D_{E_\o}^{\min}$, we need to look at a Weierstrass equation that is $\QQ$-isomorphic to $E_\o$ with integer coefficients. Let $d$ be the maximum of $\deg A$ and $\deg B$. By a change of variable $Y'=\ell_\o^{3d}Y$ and $X'=\ell_\o^{2d}X$, we obtain an integral coefficients Weierstrass equation:
$$ Y'^2=X'^3+\ell_\o^{4d}A(\o)X'+\ell_\o^{6d}B(\o)$$ with discriminant
$$\D'_{E_{\o}}:=-16 \ell_\o^{12d}(4A(\o)^3+27B(\o)^2)=\ell_\o^{12d}\D_E(\o). $$ Let us set up the following correspondence to ease our argument. If we write
$$\D_E(\ve)= \sum_{|\a|\leq d} \d_{\a} T_1^{\a_1}\ldots T_n^{\a_n}, $$ then let 
$$\widetilde{\D}_E(T_0,\ve):=\sum_{|\a|\leq d} \d_{\a} T_0^{12d-|\a|}T_1^{\a_1}\ldots T_n^{\a_n}. $$ Notice that $\widetilde{\D}_E$ is a homogeneous polynomial of degree $12d$. We have a one-to-one correspondence between
$$
\QQ^n_B(\D_E)=\{ \o \in \QQ^n \ | \ 1 <H(\o) \leq B  \text{ and }\D_E(\o) \neq 0 \} 
$$
and 
$$ \{\nu=(\nu_0,  \nu_1, \ldots, \nu_n) \in \ZZ^{n+1}_B(\widetilde{\D}_E) \ | \ \gcd(\nu_0,\nu_1, \ldots, \nu_n)=1  \text{ and }\nu_0 > 0 \} $$ via the map
\begin{align*}\o &\longmapsto \nu:=\left(\ell_\o,\frac{u_1 \ell_\o}{v_1}, \ldots, \frac{u_n \ell_\o}{v_n}\right)
\end{align*}
This correspondence gives
$$\sum_{\o \in \QQ^n_{B}(\D_E,P)} \frac{\hat{h}_{E_\o}(P_\o)}{h(\o)}=\frac{1}{2} \quad \sum_{\hidewidth\substack{ \nu \in \ZZ^{n+1}_{B}(\widetilde{\D}_E,P) \\ \gcd \nu =1 }\hidewidth } {}^{'} \quad \frac{\hat{h}_{E_{\left(\frac{\nu_1}{\nu_0}, \ldots, \frac{\nu_n}{\nu_0}\right)}}(P_{\left(\frac{\nu_1}{\nu_0}, \ldots, \frac{\nu_n}{\nu_0}\right)})}{h([\nu_0,\ldots,\nu_n])},$$ where $\gcd \nu := \gcd (\nu_0, \ldots, \nu_n)$ and the primed summation means $\nu_0 \neq 0$ with the factor $\frac{1}{2}$ taking care of the negative $\nu_0$. In order to use the inclusion-exclusion argument effectively in the later part, we need to modify the estimate on the set of $\nu$ 
for which $\widetilde{\D}_E(\nu)$ is $k$-free. Let $$\sqfr \widetilde{\D}_E(T_0,\ve):=f_E(T_0,\ve),$$ which is a homogeneous polynomial too and let
\begin{align*}&\ZZ_B^{n+1}(\widetilde{\D}_E, k, P^{nt})\\
&:=
\left\{ \nu=(\nu_0, \ldots, \nu_n) \in \ZZ_B^{n+1}(\widetilde{\D}_E,P) \ \left| \begin{array}{l} \nu_0 \neq 0, f_E(\nu) \text{ is $k$-free, } \\ P_\o \in E_\o(\QQ)_{nt} \\ \text{where } \o=(\frac{\nu_1}{\nu_0}, \ldots, \frac{\nu_n}{\nu_0}) \end{array}\right.\right\},
\end{align*} for some $k\geq 4$ big enough as in lemma \ref{k-free}. If $\a$ is the maximum of the exponents of distinct irreducible factors of $\widetilde{\D}_E(T_0,\ve)$, then for all $\nu \in \ZZ_B^{n+1}(\widetilde{\D}_E, k, P^{nt})$ with $\gcd \nu =1$, let $\o=\left(\frac{\nu_1}{\nu_0}, \ldots, \frac{\nu_n}{\nu_0}\right)$ and we have
$$\widetilde{\D}_E(\nu)=\D'_{E_\o}  \quad \text{is } k\a\text{-free.}$$ Thus, letting $N_k:=\lcm(1, \dots, k\a)$ and using the same argument as in lemma \ref{lowerB KN}, we get
$$\sum_{\o \in \QQ^n_{B}(\D_E,P)} \frac{\hat{h}_{E_\o}(P_\o)}{h(\o)}\geq \frac{1}{2} \quad \sum_{\hidewidth\substack{ \nu \in \ZZ_B^{n+1}(\widetilde{\D}_E, k, P^{nt}) \\ \gcd \nu =1 }\hidewidth } {}^{'} \quad \frac{C_1}{N_k^2} \frac{\log \D^{\min}_{E_\o}}{\log ||\nu||}.$$
Notice that $\widetilde{\D}_E(T_0,\ve)$ is not a constant times a twelfth power in $\ZZ[T_0,\ve]$, otherwise it will imply the same for $\widetilde{\D}_E(1,\ve)=\D_E(\ve)$. Just like in the proof of Proposition \ref{p1}, we write $\widetilde{\D}_E(T_0,\ve)=\b (F(T_0,\ve))^{a+12b}$, where $\b \in \ZZ$, $F(T_0,\ve)$ is primitive homogeneous in $\ZZ[T_0,\ve]$ and non-power in $\CC[T_0,\ve]$, $b\in \NN$ and $a \in \{1,2,3,\ldots, 11 \}$. Since the same property
$$0\leq \ord_p(\D^{\min}_{E_\o}) \equiv \ord_p(\widetilde{\D}_E(\nu)) \pmod{12} $$ still hold for all prime $p$ in $\QQ$, we can repeat the corresponding whole argument as in section \ref{Proof of p1} and get
\begin{equation}\sum_{\o \in \QQ^n_{B}(\D_E,P)} \frac{\hat{h}_{E_\o}(P_\o)}{h(\o)}\geq \frac{1}{2} \quad \sum_{\hidewidth\substack{  \nu \in \text{Good}_M(F,B)\\ \nu \in \ZZ_B^{n+1}(\widetilde{\D}_E, k, P^{nt}) \\ \gcd \nu =1 }\hidewidth } {}^{'} \quad \frac{C_1}{N_k^2} \frac{1}{M} + O\left(\sum_{\nu \in  \ZZ_B^{n+1}} \frac{1}{\log ||\nu||}\right) \label{inent}\end{equation}
where
$$\text{Good}_M(F,B):= \left\{\begin{array}{lll}
                   \left\{ \nu \in \ZZ^{n+1}_{B}(\widetilde{\D}_E)\ \left| \ |\sqfr(F(\nu))|^M>||\nu||\right.\right\} & \text{if } a \neq 4,8 \\ \\ \left\{ \nu \in \ZZ^{n+1}_{B}(\widetilde{\D}_E)\ \left| \ |\cufr(F(\nu))|^M>||\nu||\right.\right\} & \text{if } a = 4 , 8,
 \end{array}\right.$$ for any fixed $M>2$. We know from lemma \ref{small o} that the second term is $o(B^{n+1})$. As for the first term, we estimate it by an inclusion-exclusion argument using the M\"{o}bius function:
\begin{align} \label{mo eq}\sum_{\hidewidth\substack{  \nu \in \text{Good}_M(F,B)\\ \nu \in \ZZ_B^{n+1}(\widetilde{\D}_E, k, P^{nt}) \\ \gcd \nu =1 }\hidewidth } {}^{'} 1 \qquad &= \qquad \sum_{\hidewidth\substack{  \nu \in \text{Good}_M(F,B)\\ \nu \in \ZZ_B^{n+1}(\widetilde{\D}_E, k, P^{nt})  }\hidewidth } {}^{'}\qquad  \sum_{g| \gcd \nu} \m(g) \notag \\
&=\sum_{g=1}^B \m(g) \qquad \sum_{\hidewidth\substack{  \nu \in \text{Good}_M(F,B)\\ \nu \in \ZZ_B^{n+1}(\widetilde{\D}_E, k, P^{nt}) \\ g|\gcd(\nu) }\hidewidth } {}^{'}1.
\end{align} 
To deal with the inner sum, we have to analyse the sets of which we are summing over. Recall that $F$ is a homogeneous polynomial, so we have $F(g\nu)=g^t F(\nu)$ where $t=\deg F$ and the trivial inequalities
\begin{align*}
\sqfr(F(g\nu)) &\geq \frac{\sqfr(F(\nu))}{g},\\
\cufr(F(g\nu)) &\geq \frac{\cufr(F(\nu))}{g^2}.
\end{align*}
These imply the following inclusions:
\begin{align*}
&\text{Good}_M(F,B,g):=\left\{ \left.\nu \in \ZZ_B^{n+1}(\widetilde{\D}_E)   \ \right| \ g|\gcd(\nu)  \right\}\cap\text{Good}_M(F,B)\\
&=\left\{\begin{array}{lll}
                   g\cdot \left\{ \left.\nu \in \ZZ^{n+1}_{\frac{B}{g}}(\widetilde{\D}_E)\ \right| \ |\sqfr(F(g\nu))|^M>||g\nu||\right\} & \text{if } a \neq 4,8 \\ \\ g\cdot \left\{ \left.\nu \in \ZZ^{n+1}_{\frac{B}{g}}(\widetilde{\D}_E)\ \right| \ |\cufr(F(g\nu))|^M>||g\nu||\right\} & \text{if } a = 4 , 8
 \end{array}\right.\\
 \\
&\supseteq\left\{\begin{array}{lll}
                  g\cdot \left\{ \nu \in \ZZ^{n+1}_{\frac{B}{g}}(\widetilde{\D}_E)\ \left| \ \left|\dfrac{\sqfr(F(\nu))}{g}\right|^M>g||\nu||\right\}\right. & \text{if } a \neq 4,8 \\ \\ g\cdot \left\{ \nu \in \ZZ^{n+1}_{\frac{B}{g}}(\widetilde{\D}_E)\ \left| \ \left|\dfrac{\cufr(F(\nu))}{g^2}\right|^M>g||\nu||\right\}\right. & \text{if } a = 4 , 8,
 \end{array}\right.
\end{align*}where the notation $g\cdot S$ means $\{ g\nu | \ \nu \in S \}$ for any set $S$ of vectors. 
By Corollary \ref{CorPfree}, for $g \leq B^{\frac{1}{5}}$ and $M > 8$, we have 
\begin{equation}\text{Good}_M(F,B,g) \sim \left(2\frac{B}{g}\right)^{n+1}.  \label{GBd}\end{equation}

On the other hand, $f_E$ is also homogeneous. Let the degree of $f_E$ be $r$ and we have
\begin{align*}&\ZZ_B^{n+1}(\widetilde{\D}_E, k, P^{nt},g)\\&:=\left\{\left.\nu \in \ZZ_B^{n+1}(\widetilde{\D}_E)   \ \right| \ g|\gcd(\nu)  \right\}\cap \ZZ_B^{n+1}(\widetilde{\D}_E, k, P^{nt})\\
&=
g \cdot \left\{ \nu=(\nu_0, \ldots, \nu_n) \in \ZZ_\frac{B}{g}^{n+1}(\widetilde{\D}_E,P) \ \left| \begin{array}{l} \nu_0 \neq 0,\ g^rf_E(\nu) \text{ is $k$-free, } \\ P_\o \in E_\o(\QQ)_{nt} \\ \text{where } \o=(\frac{\nu_1}{\nu_0}, \ldots, \frac{\nu_n}{\nu_0}) \end{array}\right.\right\}.
\end{align*}For $\mu(g)\neq 0$, i.e. $g$ is squarefree, we have the inclusions
$$ g \cdot \ZZ_\frac{B}{g}^{n+1}(\widetilde{\D}_E, k-r, P^{nt})\subseteq\ZZ_B^{n+1}(\widetilde{\D}_E, k, P^{nt},g)\subseteq g \cdot \ZZ_\frac{B}{g}^{n+1}(\widetilde{\D}_E, k, P^{nt}).$$
From (\ref{GBd}), lemma \ref{k-free} and Mazur's theorem again, we have for any $\e>0$, there exists $B_\e$ such that if $\frac{B}{g} \geq B_\e$, $g \leq B^{\frac{1}{5}}$ and $\mu(g) \neq 0$ then
$$ \g_{k-r,f_E}-\e\leq  \frac{\#\left(\ZZ_B^{n+1}(\widetilde{\D}_E, k, P^{nt},g) \cap \text{Good}_M(F,B,g)\right)}{\left(2\frac{B}{g}\right)^{n+1}} \leq \g_{k,f_E}+\e. $$

It is important to remark that the implicit constants of the big $O$'s that appear in the rest of the proof depend only on $n$ and nothing else. From (\ref{mo eq}), for $B>B^\frac{5}{4}_\e$, $ \qquad \displaystyle \sum_{\hidewidth\substack{  \nu \in \text{Good}_M(F,B)\\ \nu \in \ZZ_B^{n+1}(\widetilde{\D}_E, k, P^{nt}) \\ \gcd \nu =1 }\hidewidth } {}^{'} 1 \qquad $ is bounded below by
\begin{align*}  & \sum_{g=1}^{\left\lfloor B^\frac{1}{5}\right\rfloor} \m(g) \quad \sum_{\hidewidth\substack{  \nu \in \text{Good}_M(F,B,g)\\ \nu \in \ZZ_B^{n+1}(\widetilde{\D}_E, k, P^{nt},g)}\hidewidth } {}^{'}1 \quad + \quad \sum_{g=B^\frac{1}{5}}^B \m(g) \quad \sum_{\hidewidth\substack{  \nu \in \text{Good}_M(F,B,g)\\ \nu \in \ZZ_B^{n+1}(\widetilde{\D}_E, k, P^{nt},g) }\hidewidth } {}^{'}1\\
&\geq \sum_{\substack{g=1\\ \m(g)=1}}^{B^\frac{1}{5}} \m(g) (\g_{k-r,f_E}-\e)\left( 2\frac{B}{g}\right)^{n+1}  + \sum_{\substack{g=1\\ \m(g)=-1}}^{B^\frac{1}{5}} \m(g) (\g_{k,f_E}+\e)\left( 2\frac{B}{g}\right)^{n+1} \\
&\qquad +O\left(\sum_{g=B^\frac{1}{5}}^B \left( 2\frac{B}{g}\right)^{n+1}\right)\\
&=  (\g_{k-r,f_E}-\e)(2B)^{n+1} \sum_{\substack{g=1}}^{B^\frac{1}{5}} \frac{\m(g)}{g^{n+1}}   \\ &\qquad+  (2\e+\g_{k,f_E}-\g_{k-r,f_E}) (2B)^{n+1}\sum_{\substack{g=1\\ \m(g)=-1}}^{B^\frac{1}{5}} \frac{\m(g)}{g^{n+1}}\\
&\qquad +(2B)^{n+1} O\left(\sum_{g=B^\frac{1}{5}}^B \frac{1}{g^{n+1}}\right).
\end{align*}So we get
\begin{align*}
&\liminf_{B\rightarrow \infty}\frac{1}{(2B)^{n+1}}\sum_{\o \in \QQ^n_{B}(\D_E,P)} \frac{\hat{h}_{E_\o}(P_\o)}{h(\o)} \\
&\geq \frac{1}{2}\frac{C_1}{N_k^2} \frac{1}{M}\left((\g_{k-r,f_E}-\e)\frac{1}{\z(n+1)} + (2\e+\g_{k,f_E}-\g_{k-r,f_E} )O(1) \right)
\end{align*} where $\z$ is the Riemann zeta function and one possible bound for the $O(1)$ here is $\z (n+1)$.  So 
\begin{align*}
&\liminf_{B\rightarrow \infty}\frac{2 \z(n+1)}{(2B)^{n+1}}\sum_{\o \in \QQ^n_{B}(\D_E,P)} \frac{\hat{h}_{E_\o}(P_\o)}{h(\o)} \\
&\geq \frac{C_1}{M N_k^2}\left((\g_{k-r,f_E}-\e) + (2\e+\g_{k,f_E}-\g_{k-r,f_E} ) O(1) \right).
\end{align*}
Since this holds for all $\e>0$, $M > 8$ and the same inclusion-exclusion argument will give 
$$\# \{ \o \in \QQ^n \ | \ H(\o) \leq B \} \sim \frac{(2B)^{n+1}}{2 \z (n+1)},$$ we have proven Theorem \ref{T1} with $L_1=\frac{C_1}{8 N_k^2}(\g_{k-r,f_E}+(\g_{k,f_E}-\g_{k-r,f_E} ) O(1))$. Notice that $L_1$ is positive for $k$ big enough because the sequence $\left(\g_{k,f_E}\right)_{k=1}^\infty$ is increasing and bounded above by $1$.

\section{Discussion}
Our proofs of Proposition \ref{p1} and Theorem \ref{T1} use the weakened form of Lang's height conjecture proven by Silverman mentioned in lemma \ref{lowerB KN}, which is a key tool in our proof that there is a positive density $\g_{k-r,f_E}+  (\g_{k,f_E}-\g_{k-r,f_E} ) O(1)$ of $\o \in \QQ^n(\D_E)$ such that $\log |\D_{E_\o}^{\min}|  > \left(\frac{1}{M}h(\o)+O_E(1)\right)$ and hence $ \hat{h}_{E_\o}(p) \geq \frac{C_1}{N_k^2}\left(\frac{1}{M}h(\o)+O_E(1)\right) $ for all $p \in E_\o(\QQ)_{nt}$. If we denote 
$$\m(\o):= \min \left\{\hat{h}_{E_\o}(p)\ \left| \ p \in  E_\o(\QQ)_{nt} \right. \right \},$$ then we can get a uniform upper bound of the quotient $\frac{h(\o)}{\m(\o)}$ for a positive density of $\o \in \QQ^n(\D_E)$. In view of this, we can apply this to the following theorem to say something about the integral points on $E_\o$.

\begin{theorem}\label{app}(\cite{SST} Cor.4.2) Set the following notations:

$F$ a number field.

$S$ a finite set of places of the absolute values of $F$.

$R_S(\e)=\left\{ x \in F \ | \ \sum_{v \in S}\max(-v(x),0) \geq \e h(x)\right\},$ so in particular $R_S(1)=R_S$, the ring of $S$-integers of $F$. 

$T/F$ a quasi-projective variety and $h_T$ the height on $T$ correspondind to a fixed ample divisor, chosen so that $h_T(t) \geq 1$ for all $t \in T(\overline{F})$.

$C/F$ an algebraic family of smooth, irreducible, projective curves over $T$, i.e. there is a $F$-morphism $\pi : C \rightarrow T$ which is proper and smooth of relative dimension $1$; each fiber $C_t$ is a smooth irreducible projective curve.

$J/F$ the Jacobian of $C/T$, so $J$ is an abelian acheme over $T$. Let $D \in \Div_{F}(J)$ a very ample and symmetric divisor. For each $t \in T$, the fiber $J_t$ is the Jacobian variety of the fiber $C_t$ and let
\begin{align*}\r(t)&:=\rank J_t(F), \\
\t(t)&:= \# J_t(F)_{tors}, \\
\m(t)&:=\min \left\{\hat{h}_{J,D}(p)\ \left| \ p \in  J_t(F), P \text{ non-torsion} \right. \right \}.
\end{align*}
Let $\e >0$ and $f \in F(C)$ be a non-constant rational function on $C$ with the following property:
$$ \text{the map } [f:1]:C\longrightarrow \PP^1 \text{ is a morphism.}$$ Then there exists a constant $c$ depending only on $[F:\QQ]$, $\e$, $C$, $f$ and $T$, such that for all $t \in T(F)$, the size of the set
$$ \left\{ p \in C_t(F)\ | \ f(p)\in R_s(\e) \right\}$$is at most $$ \t(t) c^{1+\#S + \r(t)}\left(\frac{h_T(t)}{\m(t)} \right)^\frac{\r(t)}{2}.$$  
\end{theorem}

\begin{corollary}\label{last co} With the setting and notations as in the proof of Theorem \ref{T1} and let $0<\d <1$ and $M>8$, then there exists a constant $c$ depending only on $E$, such that the set
$$\left\{ \o_\in \QQ^n(\D_E)\ \left| \ \#E_\o(\ZZ) \leq 16 c^{2+\rank E_\o(\QQ) } \left(\sqrt{\frac{MN_k^2}{(1-\d)C_1}} \right)^{\rank E_\o(\QQ)} \right.\right\} $$has density at least $\g_{k-r,f_E}+  (\g_{k,f_E}-\g_{k-r,f_E} ) O(1)$.
\end{corollary}
\begin{proof} We will apply Theorem \ref{app} for $F=\QQ$, $S=\{ |\cdot|_\infty \}$, $\e=1$ and so $R_S(\e)=\ZZ$. Let $T \subset \QQ^n(\D_E)$ be a quasi-projective variety such that we can define a smooth group scheme $C$ over $T$ associated to our elliptic curve $E/K$:
$$C:  Y^2Z=X^3+A(\ve)XZ^2+B(\ve)Z^3. $$
Notice that since each fiber $C_t=E_t$ is an elliptic curve, it is equal to its Jacobian $J_t$. We will use the rational function $f=x=\frac{X}{Z} \in \QQ(C)$ and we need to show that 
$$ [x:1]: C \longrightarrow \PP^1$$ is a morphism. The map is clearly defined on all points with $Z \neq 0$. From the equation for $C$, we have
$$ [x:1]=[X:Z]=[Y^2-B(\ve)Z^2 : X^2+A(\ve)Z^2].$$ Since any point of $C$ with $Z=0$ will have the form $([0:1:0], t)$, we see that $[x:1]$ will map such a point to $[1:0]$. So $[x:1]$ defines a morphims and Theorem \ref{app} says that there exists a constant $c$ depending on $E$, such that for all $\o \in T(\QQ) \subset \QQ^n(\D_E)$, we have
\begin{align}\#\left\{ p \in E_\o(\QQ)\ | \ x(p)\in \ZZ \right\}  &\leq \t(\o) c^{1+1 + \r(\o)}\left(\frac{h(\o)}{\m(\o)} \right)^\frac{\r(\o)}{2} \notag\\
&\leq 16 c^{2 + \rank E_\o(\QQ)}\left(\frac{h(\o)}{\m(\o)} \right)^\frac{\rank E_\o(\QQ)}{2}, \label{appine1}
\end{align}where we obtain the second inequality by bounding $\t(\o) \leq 16$ using Mazur's theorem.

From the proof (See inequality (\ref{inent})) of Theorem \ref{T1}, we have for any $M>8$,  there is a positive density $\g_{k-r,f_E}+  (\g_{k,f_E}-\g_{k-r,f_E} ) O(1)$ of $\o \in \QQ^n(\D_E)$ such that
$$\m(\o) > \frac{C_1}{N_k^2} \left(\frac{1}{M}h(\o)+O_E(1)\right).$$ For any $0<\d<1$, the set of bounded height
$$B_\d:=\left\{ \o \in \QQ^n \ \left| \ h(\o) \leq  \frac{M |O_E(1)|}{\d}\right. \right\} $$ is a finite set. By excluding these finite points, we still have a positive density $\g_{k-r,f_E}+  (\g_{k,f_E}-\g_{k-r,f_E} ) O(1)$ of $\o \in \QQ^n(\D_E)\backslash B_\d$ such that
\begin{equation}\m(\o) > \frac{C_1}{N_k^2} \left(\frac{1-\d}{M}h(\o)+\frac{\d}{M}h(\o)+O_E(1)\right) \geq \frac{C_1}{N_k^2} \frac{(1-\d)}{M}h(\o).\label{appine2}\end{equation}
Since $T$ is a dense Zariski open subset of $\QQ^n(\D_E)$, from inequalities (\ref{appine1}) and (\ref{appine2}), we have a positive density $\g_{k-r,f_E}+  (\g_{k,f_E}-\g_{k-r,f_E} ) O(1)$ of $\o \in (\QQ^n(\D_E)\backslash B_\d) \cap T(\QQ)$ such that
\begin{align*}\#E_\o(\ZZ) \leq \#\left\{ p \in E_\o(\QQ)\ | \ x(p)\in \ZZ \right\} &\leq 16 c^{2 + \rank E_\o(\QQ)}\left(\frac{h(\o)}{\m(\o)} \right)^\frac{\rank E_\o(\QQ)}{2}  \\
& \leq 16 c^{2 + \rank E_\o(\QQ)}\left(\frac{M N_k^2}{(1-\d)C_1 } \right)^\frac{\rank E_\o(\QQ)}{2},
\end{align*}which completes the proof of the corollary.
\end{proof}
We remark that if the Lang's conjecture is true, then we can improve both Proposition \ref{p1} and Theorem \ref{T1} to $L_2=\frac{C_1}{2}$ and $L_1=\frac{C_1}{8}$, independent of $E$. Also, corollary \ref{last co} will be improved to density $1$.

One might be interested to ask whether we can generalize our initial setting of $\QQ$ to any number field $F$. In order to do that, we first have to replace $\ZZ$ to $F$ integers $\Ocal_F$ in Proposition \ref{p1} and scrutinize all the lemmas used in the proof to see whether they are still valid in $F$. Lemma \ref{lowerB KN} can be easily generalized to $F$ as both the Silverman \cite{SLB} and Kodaira-N\'eron Theorems \cite{SAT} were originally proven for number fields. 
Further, lemmas \ref{small o}, \ref{Mason Co} generalize immediately, Mazur's theorem also has its generalized counterpart, Merel's Theorem. Alternatively, we can use the following Masser's bound (we thank the referee for pointing this out). Using methods from transcendence theory, Masser obtained the upper bound (\cite{MCP} Corollary 2)
 $$\# E(K)_{tor} \leq C_k\sqrt{ h([1, g_1, g_2])} [K:k]\left(h([1, g_1, g_2])+\log [K:k]\right)$$for elliptic curve $E/k \ :\ y^2=4x^3-g_1x-g_3$, where $C_k$ is an effective constant that depends only on the number field $k$ and $K/k$ is any finite field extension. Hence, applying this to our setting over the number field $F$, we can obtain easily that for all $\nu \in (\Ocal_F)^n_B$,
$$\# E_\nu(F)_{tor}\leq C_F (h([1,4A(\nu), 4B(\nu)]))^\frac{3}{2} \leq C'_F (\log B)^\frac{3}{2} $$ where $C_F'$ is an effective constant that depends on $F$ and the polynomials $A(\ve),B(\ve)$. This is sufficient for our application as it gives us the bound 
$$\# \{\nu \in \Ocal_F^n \ | \ H(\nu) \leq B \text{ and } P_\nu \text{ is torsion}  \}=O(B^{n-1}(\log B)^\frac{3}{2})=o(B^n).$$ Besides having the advantage of a computable effective constant, Masser's bound is also true for general abelian varieties (\cite{MSF} Main Theorem and Scholium 2).

What are left to be worked on are lemmas \ref{k-free} and \ref{Good density}. Another, and possibly more interesting problem is to prove convergence of the average, or even better, to prove the average converges to $\hat{h}_{E}(P)$.


\subsection*{Acknowledgements}
I would like to thank my advisor, Joseph Silverman, for many enlightening discussions. Also, many thanks to the referee for the valuble and insightful comments and suggestions.

\end{document}